\documentclass{amsart}

\usepackage{amscd,amsmath,xypic,amssymb,combelow,tikz-cd,etoolbox,calligra,mathrsfs,enumitem,mathtools,epigraph}

\usepackage[russian,english]{babel}

\makeatletter
\patchcmd{\@settitle}{\uppercasenonmath\@title}{}{}{}
\makeatother

\xyoption{all}
\CompileMatrices

\makeatletter
\@addtoreset{equation}{section}
\makeatother

\newtheorem{theorem}[subsection]{Theorem}

\newtheorem{proposition}[subsection]{Proposition}

\newtheorem{corollary}[subsection]{Corollary}
\newtheorem{conjecture}[subsection]{Conjecture}

\newtheorem{definition}[subsection]{Definition}

\newtheorem{claim}[subsection]{Claim}
\newtheorem{problem}[subsection]{Problem}
\newtheorem{example}[subsection]{Example}
\newtheorem{remark}[subsection]{Remark}

\def\loccitt{\emph{loc. cit.}}
\def\loccit{\emph{loc. cit. }}

\def\fg{{\mathfrak{g}}}

\def\fsl{{\mathfrak{sl}}}
\def\fgl{{\mathfrak{gl}}}

\def\fZ{{\mathfrak{Z}}}

\def\BC{{\mathbb{C}}}

\def\BN{{\mathbb{N}}}
\def\BF{{\mathbb{F}}}
\def\BP{{\mathbb{P}}}

\def\BQ{{\mathbb{Q}}}
\def\BZ{{\mathbb{Z}}}

\def\CA{{\mathcal{A}}}
\def\CB{{\mathcal{B}}}

\def\CL{{\mathcal{L}}}

\def\CN{{\mathcal{N}}}
\def\CO{{\mathcal{O}}}

\def\CR{{\mathcal{R}}}
\def\CS{{\mathcal{S}}}

\def\CV{{\mathcal{V}}}

\def\Hom{\textrm{Hom}}

\def\vs{\varsigma}

\def\and{\textrm{ }\&\textrm{ }}

\def\sym{\textrm{sym}}

\def\tzeta{{\widetilde{\zeta}}}

\def\nn{{{\BN}}^I}
\def\zz{{{\BZ}}^I}
\def\qq{{{\BQ}}^I}
\def\qqp{{{\BQ}}_+^I}

\def\bk{\boldsymbol{k}}
\def\bl{\boldsymbol{l}}
\def\bm{\boldsymbol{m}}
\def\bn{\boldsymbol{n}}
\def\bv{\boldsymbol{v}}
\def\bw{\boldsymbol{w}}
\def\bz{\boldsymbol{z}}
\def\bX{\boldsymbol{X}}
\def\bW{\boldsymbol{W}}
\def\bZ{\boldsymbol{Z}}
\def\bs{{\boldsymbol{\vs}}}
\def\bth{{\boldsymbol{\theta}}}
\def\b0{{\boldsymbol{0}}}

\def\hdeg{\text{hdeg }}
\def\vdeg{\text{vdeg }}

\def\op{\text{op}}

\def\oii{\overrightarrow{ii}}
\def\oij{\overrightarrow{ij}}
\def\oji{\overrightarrow{ji}}
\def\loc{\text{loc}}

\begin{document}

\title[Shuffle algebras for quivers and $R$-matrices]{\Large{\textbf{Shuffle algebras for quivers and $R$-matrices}}}

\author[Andrei Negu\cb t]{Andrei Negu\cb t}
\address{MIT, Department of Mathematics, Cambridge, MA, USA}
\address{Simion Stoilow Institute of Mathematics, Bucharest, Romania}
\email{andrei.negut@gmail.com}

\maketitle

\begin{abstract} 

We define slope subalgebras in the shuffle algebra associated to a (doubled) quiver, thus yielding a factorization of the universal $R$-matrix of the double of the shuffle algebra in question. We conjecture that this factorization matches the one defined by \cite{AO,MO,O1,O2,OS} using Nakajima quiver varieties.

\end{abstract}

$$$$

\section{Introduction}

\medskip

Fix a quiver $Q$ with vertex set $I$ and edge set $E$; edge loops and multiple edges are allowed. We consider a certain Hopf algebra:
$$
\CA = \CA^+ \otimes (\text{Cartan subalgebra}) \otimes \CA^-
$$
where $\CA^+$ is the shuffle algebra associated to the (double of the) quiver $Q$ and $\CA^-$ is its opposite. When $Q$ is a finite (respectively affine) Dynkin diagram, the algebra $\CA$ is the quantum loop (respectively quantum toroidal) algebra. In general, the shuffle algebra $\CA^+$ matches the localized $K$-theoretic Hall algebra of the quiver $Q$ (\cite{N}). \\

\noindent The main purpose of the present paper is to define and study \textbf{slope} subalgebras:
$$
\CB_{\bm}^{\pm} \subset \CA^\pm
$$
for any $\bm \in \qq$, and produce a Hopf algebra:
$$
\CB_{\bm} = \CB^+_{\bm} \otimes (\text{Cartan subalgebra}) \otimes \CB^-_{\bm}
$$
For non-trivial reasons, there exist inclusions $\CB_{\bm} \subset \CA$ which preserve the product and the Hopf pairing, but not the coproduct and antipode. As $\CA$ and $\CB_{\bm}$ arise as Drinfeld doubles, we may consider their universal $R$-matrices: \footnote{The symbol $\widehat{\otimes}$ refers to the fact that the universal $R$-matrices lie in certain completions of the algebras in question, as they are given by infinite sums. Meanwhile, the primes refer to the fact that $\CR'$ is only the ``off-diagonal" part of the universal $R$-matrix, see \eqref{eqn:univ r} and \eqref{eqn:univ r sub}}
$$
\CR' \in \CA \ \widehat{\otimes} \ \CA \qquad \text{and} \qquad \CR'_{\bm} \in \CB_{\bm} \ \widehat{\otimes} \ \CB_{\bm}
$$
Our main result, proved by combining Corollaries \ref{cor:iso} and \ref{cor:r-matrix}, is the following. \\

\begin{theorem}
\label{thm:intro}

For any $\bm \in \qq$ and $\bth \in \qqp$, the multiplication map induces an isomorphism (the arrow $\rightarrow$ refers to taking the product in increasing order of $r$):
\begin{equation}
\label{eqn:intro 1}
\bigotimes_{r \in \BQ}^{\rightarrow} \CB_{\bm+r\bth}^\pm \xrightarrow{\sim}\CA^\pm
\end{equation}
which preserves the Hopf pairings on the two sides, and thus leads to a factorization:
\begin{equation}
\label{eqn:intro 2}
\CR' = \prod_{r \in \BQ}^{\rightarrow} \CR'_{\bm+r\bth}
\end{equation}
of the (off-diagonal part of the) universal $R$-matrix. \\

\end{theorem}

\noindent When $Q$ is a cyclic quiver, Theorem \ref{thm:intro} was proved in \cite{Thesis,Cyc}. The isomorphism \eqref{eqn:intro 1} is inspired by the one constructed by Burban-Schiffmann \cite{BS} in the elliptic Hall algebra (which is isomorphic to $\CA^+$ when $Q$ is the Jordan quiver, namely one vertex and one loop). Meanwhile, the product formula \eqref{eqn:intro 2} generalizes well-known formulas for $R$-matrices of finite and affine type quantum groups \cite{D,KR,KT,LS,LSS,R}. \\

\noindent In Section \ref{sec:shuffle}, we recall general facts about the shuffle algebra $\CA^+$. In Section \ref{sec:slope}, we define the slope subalgebras $\CB_{\bm}$ and prove Theorem \ref{thm:intro}. In Section \ref{sec:geometry}, we present connections (as well as conjectures and open questions) between our slope subalgebras and other concepts in the field such as Kac polynomials, cohomological and $K$-theoretic Hall algebras (particularly the connection between $\CB_{\b0}$ and the Lie algebra of BPS states studied in \cite{Dav, DM}) and the conjectural connection between our formulas \eqref{eqn:intro 1} and \eqref{eqn:intro 2} and the analogous formulas for quantum groups defined via geometric $R$-matrices \cite{AO,MO,O1,O2,OS} in the context of Nakajima quiver varieties. \\

\noindent It is likely that Theorem \ref{thm:intro} can be generalized to the case of quivers with potential, although working out all the details would probably be a very non-trivial and interesting task (see \cite{Pad} for the setting of such a generalization; note, however, that Theorem 6.3 of \loccit provides an isomorphism of a different nature from \eqref{eqn:intro 1}). \\

\noindent I would like to thank Andrei Okounkov, Olivier Schiffmann and Alexander Tsymbaliuk for many interesting discussions about $R$-matrices and much more. I gratefully acknowledge NSF grants DMS-$1760264$ and DMS-$1845034$, as well as support from the Alfred P.\ Sloan Foundation and the MIT Research Support Committee.\\

\section{The shuffle algebra of a (doubled) quiver} 
\label{sec:shuffle}

\medskip

\subsection{} A quiver is a finite oriented graph $Q$ with vertex set $I$ and edge set $E$; edge loops and multiple edges are allowed. We will work over the field:
$$
\BF = \BQ(q,t_e)_{e \in E}
$$
We will write elements of $\nn$ as $\bn = (n_i \geq 0)_{i \in I}$ \footnote{Although non-standard, it will be convenient for us to include 0 in the set $\BN$}. For such an $\bn$, let us define:
$$
\bn! = \prod_{i\in I} n_i!
$$
Consider the vector space:
\begin{equation}
\label{eqn:big shuffle}
\CV = \bigoplus_{\bn = (n_i)_{i \in I} \in \nn} \BF[\dots,z^{\pm 1}_{i1},\dots,z^{\pm 1}_{in_i},\dots]^{\sym}
\end{equation}
where ``$\sym$" refers to Laurent polynomials which are symmetric in $z_{i1},\dots,z_{in_i}$ for each $i \in I$ separately. We will make $\CV$ into an associative algebra using the following \textbf{shuffle product} (it originated with a construction of \cite{FO} involving elliptic algebras, though the setting at hand is closer to the one studied in \cite{E,FHHSY,SV} and other works):
\begin{equation}
\label{eqn:shuf prod}
F(\dots,z_{i1},\dots,z_{in_i},\dots) * F'(\dots,z_{i1},\dots,z_{in'_i},\dots) = 
\end{equation}
$$
\text{Sym} \left[\frac {F(\dots,z_{i1},\dots,z_{in_i},\dots) F'(\dots,z_{i,n_i+1},\dots,z_{i,n_i+n_i'},\dots)}{\bn! \cdot \bn'!}\mathop{\prod^{i,j \in I}_{1 \leq a \leq n_i}}_{n_{j} < b \leq n_{j}+n'_{j}} \zeta_{ij} \left( \frac {z_{ia}}{z_{jb}} \right) \right]
$$
where ``Sym" denotes symmetrization with respect to the variables $z_{i1},\dots,z_{i,n_i+n_i'}$ for each $i \in I$ separately, and for any $i,j \in I$ we define the following function:
\begin{equation}
\label{eqn:def zeta}
\zeta_{ij}(x) = \left(\frac {1-xq^{-1}}{1-x} \right)^{\delta_j^i} \prod_{e = \oij \in E} \left(\frac 1{t_e} - x \right) \prod_{e = \oji \in E} \left(1 - \frac {t_e}{qx} \right)
\end{equation}
Note that although the right-hand side of \eqref{eqn:shuf prod} seemingly has simple poles at $z_{ia} - z_{ib}$ for all $i \in I$ and all $a < b$, these poles vanish when taking the symmetrization, as the orders of such poles in a symmetric rational function must be even. \\


\begin{definition}
\label{def:shuffle}

The \textbf{shuffle algebra} is defined as the subset:
$$
\CS \subset \CV
$$
of Laurent polynomials $F(\dots, z_{i1}, \dots, z_{in_i}, \dots)$ that satisfy the ``wheel conditions":
\begin{equation}
\label{eqn:wheel}
F \Big|_{z_{ia} = \frac {qz_{jb}}{t_e} = q z_{ic}} = F \Big|_{z_{ja} = t_e z_{ib} = q z_{jc} } = 0
\end{equation}
for all edges $e = \oij$ and all $a \neq c$ (and further $a \neq b \neq c$ if $i = j$). \\

\end{definition}

\noindent It is easy to show that $\CS$ is a subalgebra of $\CV$, i.e. that it is closed under the shuffle product (see \cite[Proposition 2.1]{Shuf} for the proof in the particular case of the Jordan quiver, which already incorporates all the ideas that one needs in the general case). \\

\begin{theorem}
\label{thm:generate}

(\cite[Theorem 1.2]{N}) As an $\BF$-algebra, $\CS$ is generated by $\{z_{i1}^d\}_{i \in I}^{d \in \BZ}$. \\

\end{theorem}





\subsection{} The algebra $\CS$ is $\nn \times \BZ$ graded via:
\begin{equation}
\label{eqn:degree plus}
\deg F = (\bn, d)
\end{equation}
if $F$ lies in the $\bn$-th direct summand of \eqref{eqn:big shuffle}, and has homogeneous degree $d$. The components of the degree will be called ``horizontal" and ``vertical", respectively:
\begin{equation}
\label{eqn:hor and vert}
\hdeg F = \bn, \qquad \vdeg F = d
\end{equation}
We will denote the graded pieces of the shuffle algebra by:
\begin{equation}
\label{eqn:graded pieces plus}
\CS = \bigoplus_{\bn \in \nn} \CS_{\bn} = \bigoplus_{(\bn,d) \in \nn \times \BZ} \CS_{\bn,d}
\end{equation}
For any $\bk \in \zz$, we have a \textbf{shift automorphism}:
\begin{equation}
\label{eqn:shift plus}
\CS \xrightarrow{\tau_{\bk}} \CS, \qquad F(\dots,z_{ia},\dots) \mapsto F(\dots,z_{ia},\dots) \prod_{i \in I, a \geq 1} z_{ia}^{k_i}
\end{equation}
The notions above also apply to the opposite algebra $\CS^{\op}$, although we make slightly different conventions. For one thing, we set the grading on $\CS^{\op}$ to:
\begin{equation}
\label{eqn:degree minus}
\deg G = (-\bn, d)
\end{equation}
if $G$ lies in the $\bn$-th direct summand of \eqref{eqn:big shuffle}, and has homogeneous degree $d$. As for the analogue of the shift automorphism \eqref{eqn:shift plus}, we make the following convention:
\begin{equation}
\label{eqn:shift minus}
\CS^{\op} \xrightarrow{\tau_{\bk}} \CS^{\op}, \qquad G(\dots,z_{ia},\dots) \mapsto G(\dots,z_{ia},\dots) \prod_{i \in I, a \geq 1} z_{ia}^{-k_i}
\end{equation}

\subsection{} 
\label{sub:algebra}

We will now recall the well-known Hopf algebra structure on the shuffle algebra, see \cite{Shuf, SV, YZ} for incarnations of this construction in settings such as ours. As usually, the Hopf algebra is actually the double extended shuffle algebra, namely: \footnote{It is possible to enlarge $\CA$ by introducing a central element which measures the failure of the $h^+$'s and the $h^-$'s to commute, but we will not need it in the present paper}
\begin{equation}
\label{eqn:double}
\CA = \CS \otimes \BF [h_{i,\pm 0}^{\pm 1}, h_{i,\pm 1}, h_{i,\pm 2}, \dots ]_{i \in I} \otimes \CS^{\op} \Big / \text{relations \eqref{eqn:rel double 1}--\eqref{eqn:rel double 3}}
\end{equation}
Since the algebras $\CS$ and $\CS^{\op}$ are generated by:
$$
e_{i,d} = z_{i1}^d \in \CS \qquad \text{and} \qquad f_{i,d} = z_{i1}^d \in \CS^{\op}
$$
it suffices to present the defining relations, as well as the Hopf algebra structure, on the generators. More precisely, if we package the generators into formal series:
\begin{equation}
\label{eqn:formal series}
e_i(z) = \sum_{d\in \BZ} \frac {e_{i,d}}{z^d}, \qquad f_i(z) = \sum_{d\in \BZ} \frac {f_{i,d}}{z^d}, \qquad 
h^\pm_i(w) = \sum_{d = 0}^{\infty} \frac {h_{i,\pm d}}{w^{\pm d}}
\end{equation}
then we set:
\begin{align}
&e_i(z) h^\pm_j(w) = h^\pm_j(w) e_i(z) \frac {\zeta_{ij} \left( \frac z{w} \right)}{\zeta_{ji} \left( \frac {w}z \right)} \label{eqn:rel double 1} \\
&f_i(z) h^\pm_j(w) = h^\pm_j(w) f_i(z) \frac {\zeta_{ji} \left( \frac {w}z \right)}{\zeta_{ij} \left( \frac z{w} \right)} \label{eqn:rel double 2}
\end{align}
(the rational functions in the right-hand sides of the expressions above are expanded as power series in $w^{\mp 1}$) and:
\begin{equation}
\label{eqn:rel double 3}
[e_{i,d}, f_{j,k}] =  \delta_j^i \cdot \gamma_i \begin{cases} - h_{i,d+k} &\text{if } d+k > 0 \\ h_{i,-0} - h_{i,+0} &\text{if } d+k = 0 \\ h_{i,d+k} &\text{if } d+k<0 \end{cases}
\end{equation}
where:
\begin{equation}
\label{eqn:constant}
\gamma_i = \frac {\prod_{e = \oii} \left[ \left(\frac 1{t_e}-1 \right) \left(1 - \frac {t_e}q \right) \right]}{1 - \frac 1q}
\end{equation}
\footnote{This differs slightly from the conventions of \cite{N}, where the constant $\gamma_i$ did not appear in the analogue of \eqref{eqn:rel double 3}; this can be explained by simply rescaling the generators $f_{i,d}$ by $\gamma_i$} It is easy to see that the grading \eqref{eqn:degree plus}, \eqref{eqn:degree minus} extends to the whole of $\CA$, by setting: 
$$
\deg h_{i,\pm d} = (0,\pm d)
$$
for all $i \in I$, $d \geq 0$. The shift automorphisms \eqref{eqn:shift plus}, \eqref{eqn:shift minus} extend to automorphisms: 
\begin{equation}
\label{eqn:shift}
\tau_{\bk} : \CA \rightarrow \CA
\end{equation}
by setting $\tau_{\bk}(h_{i,\pm d}) = h_{i,\pm d}$ for all $i\in I$ and $d \in \BN$. \\

\subsection{} 
\label{sub:cop}

To write down the (topological) coproduct on $\CA$, consider the subalgebras:
$$
\CA^+ = \CS \qquad \text{and} \qquad \CA^- = \CS^{\op}
$$
and the extended subalgebras:
\begin{align*}
&\CA^{\geq} = \CA^+ \otimes \BF [h_{i,+0}^{\pm 1}, h_{i,1}, h_{i,2}, \dots ]_{i \in I}\\
&\CA^{\leq} = \CA^- \otimes \BF [h_{i,-0}^{\pm 1}, h_{i,-1}, h_{i,-2}, \dots ]_{i \in I}
\end{align*}
The reason for these extended subalgebras is that $\CA^+$ (respectively $\CA^-$) does not admit a coproduct, but $\CA^{\geq}$ (respectively $\CA^{\leq}$) does, according to the formulas:
\begin{equation}
\label{eqn:cop 1}
\Delta \left(h^{\pm}_i(z) \right) = h^{\pm}_i(z) \otimes h^{\pm}_i(z)  
\end{equation}
\begin{equation}
\label{eqn:cop 2}
\Delta \left(e_i(z) \right) = e_i(z) \otimes 1 + h^+_i(z) \otimes e_i(z)  
\end{equation}
\begin{equation}
\label{eqn:cop 3}
\Delta \left(f_i(z) \right) = f_i(z) \otimes h^-_i(z) + 1 \otimes f_i(z)  
\end{equation}
There are unique antipode maps $S  : \CA^{\geq} \rightarrow \CA^{\geq}$ and $S  : \CA^{\leq} \rightarrow \CA^{\leq}$ which are determined by the topological coproducts above, and so we leave their computation to the interested reader (the antipode will not feature in the present paper). \\

\subsection{} 
\label{sub:double}

It is straightforward to show that the Hopf algebra structures on $\CA^{\geq}$ and $\CA^{\leq}$ defined above extend to the entire $\CA$. An alternative way to see this is to note that $\CA$ is the Drinfeld double of $\CA^{\geq}$ and $\CA^{\leq}$. Indeed, consider the Hopf pairing:
\begin{equation}
\label{eqn:hopf pairing}
\langle \cdot , \cdot \rangle : \CA^{\geq} \otimes \CA^{\leq} \longrightarrow \BF
\end{equation}
which is defined by the following formulas:
\begin{equation}
\label{eqn:pairing 1} 
\langle h^+_i(z), h^-_j(w)  \rangle = \frac {\zeta_{ij} \left(\frac zw \right)}{\zeta_{ji} \left(\frac wz \right)} 
\end{equation}
(the right-hand side is expanded as $|z| \gg |w|$) and:
\begin{equation}
\label{eqn:pairing 2} 
\langle e_{i,d}, f_{j,k} \rangle = \delta_j^i \gamma_i \delta_{d+k}^0
\end{equation}
All other pairings between the $e$'s, $f$'s and $h$'s vanish, from where we deduce that the pairing \eqref{eqn:hopf pairing} only pairs non-trivially elements of opposite degrees. From \eqref{eqn:pairing 1}--\eqref{eqn:pairing 2} one can then deduce the pairing on any elements by applying the properties:
\begin{align}
&\langle a,b_1b_2 \rangle = \langle \Delta(a), b_1 \otimes b_2 \rangle \label{eqn:bialg 1} \\
&\langle a_1a_2,b \rangle = \langle a_1 \otimes a_2, \Delta^{\op}(b) \rangle \label{eqn:bialg 2}
\end{align}
for all $a,a_1,a_2 \in \CA^{\geq}$ and $b,b_1,b_2 \in \CA^{\leq}$. We remark that the pairing also satisfies the following property with respect to the antipode:
$$
\langle S(a), S(b) \rangle = \langle a,b \rangle
$$
for all $a \in \CA^{\geq}$ and $b \in \CA^{\leq}$, but we will not need this fact in the present paper. The pairing \eqref{eqn:hopf pairing} was shown to be non-degenerate in \cite[Proposition 3.3]{N}, although this is also easily seen from formulas \eqref{eqn:pairing plus}--\eqref{eqn:pairing minus} below. Therefore, one can make the vector space:
\begin{equation}
\label{eqn:double}
\CA = \CA^{\geq} \otimes \CA^{\leq}
\end{equation}
into a Hopf algebra using the well-known Drinfeld double construction, as follows. First, one makes \eqref{eqn:double} into an algebra by requiring that $\CA^{\geq} = \CA^{\geq} \otimes 1 \subset \CA$ and  $\CA^{\leq} = 1 \otimes \CA^{\leq} \subset \CA$ be algebra homomorphisms, and the multiplication of elements coming from the two tensor factors in \eqref{eqn:double} be governed by the following relation:
\begin{equation}
\label{eqn:double formula}
a_1b_1 \langle a_2,b_2\rangle = b_2 a_2 \langle a_1,b_1\rangle 
\end{equation}
\footnote{In the formula above, we use Sweedler notation $\Delta(a) = a_1 \otimes a_2$ and $\Delta(b) = b_1 \otimes b_2$ for the coproduct, with the summation sign being implied} for any $a \in \CA^{\geq} \subset \CA$, $b \in \CA^{\leq} \subset \CA$. It is straightforward to show that the resulting algebra structure on $\CA$ of \eqref{eqn:double} matches the one introduced in Subsection \ref{sub:algebra}. As for the coalgebra structure and antipode on \eqref{eqn:double}, they are uniquely determined by the respective structures on the two tensor factors $\CA^{\geq}$ and $\CA^{\leq}$, and multiplicativity. \\

\subsection{}
\label{sub:r-matrix}

Since the Hopf algebra $\CA$ is a Drinfeld double, it has a universal $R$-matrix:
$$
\CR \in \CA^{\geq} \ \widehat{\otimes} \ \CA^{\leq} \subset \CA \ \widehat{\otimes} \ \CA
$$
(the completion is necessary because our coproduct is topological). Specifically, $\CR$ is the canonical tensor of the pairing \eqref{eqn:hopf pairing}, and it takes the form:
\begin{equation}
\label{eqn:univ r}
\CR = \CR' \cdot \Big[\text{a sum of products involving the }h_{i,\pm d} \Big]
\end{equation}
\footnote{For a survey of the formula above in the particular case of the Jordan quiver, we refer the reader to \cite{R-matrix}, where we recall the standard difficulties in properly defining the product in \eqref{eqn:univ r}} where $\CR'$ is the canonical tensor of the restriction of the pairing \eqref{eqn:hopf pairing} to:
\begin{equation}
\label{eqn:pairing positive}
\langle \cdot , \cdot \rangle : \CA^+ \otimes \CA^- \longrightarrow \BF
\end{equation}
In other words, we have: 
\begin{equation}
\label{eqn:universal r-matrix}
\CR' = 1 + \sum_{i \in I} \sum_{d \in \BZ} \frac {e_{i,d} \otimes f_{i,-d}}{\gamma_i} + \dots
\end{equation}
where the ellipsis denotes terms which are quadratic, cubic, etc, in the $e$'s and the $f$'s. In what follows, we will construct a factorization of $\CR'$ as an infinite product of $R$-matrices arising from ``slope subalgebras", generalizing the treatment of cyclic quivers in \cite{Thesis, Cyc}. Such factorizations are inspired by the analogous constructions pertaining to quantum groups from \cite{D, KR, KT, LS, LSS, R} (which coincide with our construction for simply laced quantum affine groups) and with the constructions of geometric $R$-matrices from \cite{AO, MO, O1, O2, OS} (see Section \ref{sec:geometry}). \\

\subsection{}

In what follows, we will need to present the bialgebra structure of Subsections \ref{sub:algebra}, \ref{sub:cop}, \ref{sub:double} in shuffle algebra language. More precisely, Theorem \ref{thm:generate} implies that formulas \eqref{eqn:rel double 1}--\eqref{eqn:rel double 2}, \eqref{eqn:cop 2}--\eqref{eqn:cop 3} and \eqref{eqn:pairing 2} extend from the generators $e_{i,d}$ (respectively $f_{i,d}$) to the entire shuffle algebra $\CS = \CA^+$ (respectively $\CS^{\op} = \CA^-$). All the statements in the present Subsection are straighforward, and left as exercises to the interested reader (equivalently, they were proved in Sections 3 and 4 of \cite{N}). Formulas \eqref{eqn:rel double 1}--\eqref{eqn:rel double 2} imply that:
\begin{align}
&F h^\pm_j(w) = h^\pm_j(w) F \prod^{i\in I}_{1 \leq a \leq n_i} \frac {\zeta_{ij} \left( \frac {z_{ia}}{w} \right)}{\zeta_{ji} \left( \frac {w}{z_{ia}} \right)} \label{eqn:rel double 1 recast} \\
&G h^\pm_j(w) = h^\pm_j(w) G \prod^{i\in I}_{1 \leq a \leq n_i} \frac {\zeta_{ji} \left( \frac {w}{z_{ia}} \right)}{\zeta_{ij} \left( \frac {z_{ia}}{w} \right)} \label{eqn:rel double 2 recast}
\end{align}
for any $F \in \CS_{\bn}$ and any $G \in \CS^{\op}_{\bn}$ (the rational functions in the right-hand sides of the formulas above are expanded as power series in $w^{\mp 1}$). In particular, by extracting the coefficient of $w^0$ from the formulas above, we obtain the following:
\begin{align}
&F h_{j,+0}  = h_{j,+0} F \cdot \prod_{i \in I} \left(q^{\delta_j^i} \prod_{e = \oij} \frac 1{t_e} \prod_{e = \oji} \frac {t_e}q \right)^{n_i}\label{eqn:rel ext plus} \\
&G h_{j,-0}  = h_{j,-0} G \cdot \prod_{i \in I} \left(\frac 1{q^{\delta_j^i}} \prod_{e = \oij} \frac q{t_e} \prod_{e = \oji} t_e \right)^{n_i}\label{eqn:rel ext minus}
\end{align}
As for formulas \eqref{eqn:cop 2}--\eqref{eqn:cop 3}, they imply the following:
\begin{align}
&\Delta (F) = \sum_{\{0 \leq k_i \leq n_i\}_{i\in I}} \frac {\prod^{j \in I}_{k_j < b \leq n_j} h^+_j(z_{jb}) F(\dots, z_{i1},\dots , z_{ik_i} \otimes z_{i,k_i+1}, \dots, z_{in_i},\dots)}{\prod^{i \in I}_{1\leq a \leq k_i} \prod^{j \in I}_{k_j < b \leq n_j} \zeta_{ji} \left( \frac {z_{jb}}{z_{ia}} \right)} \label{eqn:cop plus} \\
&\Delta(G) = \sum_{\{0 \leq k_i \leq n_i\}_{i\in I}} \frac {G(\dots, z_{i1},\dots , z_{ik_i} \otimes z_{i,k_i+1}, \dots, z_{in_i},\dots) \prod^{i \in I}_{1 \leq a \leq k_i} h^-_i(z_{ia}) }{\prod^{i \in I}_{1\leq a \leq k_i} \prod^{i \in I}_{k_j < b \leq n_j} \zeta_{ij} \left( \frac {z_{ia}}{z_{jb}} \right)} \label{eqn:cop minus}
\end{align}
for any $F \in \CS_{\bn}$ and $G \in \CS^{\op}_{\bn}$. To make sense of the right-hand side of formulas \eqref{eqn:cop plus}--\eqref{eqn:cop minus}, we expand the denominator as a power series in the range $|z_{ia}| \ll |z_{jb}|$, and place all the powers of $z_{ia}$ to the left of the $\otimes$ sign and all the powers of $z_{jb}$ to the right of the $\otimes$ sign (for all $i,j \in I$, $1 \leq a \leq k_i$, $k_j < b \leq n_j$). \\

\noindent Finally, formulas \eqref{eqn:pairing 1}--\eqref{eqn:pairing 2} together with the defining properties \eqref{eqn:bialg 1}--\eqref{eqn:bialg 2} of a bialgebra pairing imply that (see formulas (3.2) and (3.30) of \cite{N}, respectively):
\begin{align}
&\langle F, f_{i_1,d_1} * \dots * f_{i_n,d_n} \rangle = \int_{|z_1| \ll \dots \ll |z_n|} \frac {z_1^{d_1} \dots z_n^{d_n} F(z_1,\dots,z_n)}{\prod_{1\leq a < b \leq n} \zeta_{i_ai_b} \left(\frac {z_a}{z_b} \right)} \prod_{a=1}^n \frac {dz_a}{2\pi i z_a}
\label{eqn:pairing plus} \\
&\langle e_{i_1,d_1} * \dots * e_{i_n,d_n}, G \rangle = \int_{|z_1| \gg \dots \gg |z_n|} \frac {z_1^{d_1} \dots z_n^{d_n} G(z_1,\dots,z_n)}{\prod_{1\leq a < b \leq n} \zeta_{i_bi_a} \left(\frac {z_b}{z_a} \right)} \prod_{a=1}^n \frac {dz_a}{2\pi i z_a} \label{eqn:pairing minus}
\end{align}
for any $F \in \CS$ (respectively $G \in \CS^{\op}$) and any $i_1,\dots,i_n \in I$, $d_1,\dots,d_n \in \BZ$ such that the shuffle elements being paired in any $\langle \cdot, \cdot \rangle$ of \eqref{eqn:pairing plus}--\eqref{eqn:pairing minus} have opposite degrees. In order for formula \eqref{eqn:pairing plus} to make sense, one needs to plug the variable $z_a$ into a variable of the form $z_{i_a \bullet}$ of $F$, where the choice of $\bullet$ does not matter due to the symmetry of $F$. The analogous remark applies to \eqref{eqn:pairing minus}. \\

\section{Slope subalgebras and factorizations of $R$-matrices}
\label{sec:slope}

\medskip

\subsection{}

We will consider $\nn \subset \zz \subset \qq$. Recall that $\nn$ includes the element $\b0 = (0,\dots,0)$, according to our convention that $0 \in \BN$, as well as the following elements:
$$
\bs^i = (\underbrace{0,\dots,0,1,0,\dots,0}_{1 \text{ on }i\text{-th position}})
$$
$\forall i \in I$. We will consider two operations on $\nn \subset \zz \subset \qq$, namely the dot product:
\begin{equation}
\label{eqn:dot}
\bk \cdot \bl = \sum_{i \in I} k_i l_i
\end{equation}
and the bilinear form:
\begin{equation}
\label{eqn:form}
\langle \bk , \bl \rangle = \sum_{i,j \in I} k_i l_j \#_{\oij}
\end{equation}
for any $\bk = (k_i)_{i \in I}$ and $\bl = (l_i)_{i \in I}$. In the formula above, we write:
\begin{equation}
\label{eqn:number of arrows}
\#_{\oij} = \text{the number of arrows }\oij \text { in } Q
\end{equation}
For any $\bk = (k_i)_{i\in I}$ and $\bn = (n_i)_{i \in I}$, we will write: 
\begin{equation}
\label{eqn:ineq deg}
\b0 \leq \bk \leq \bn 
\end{equation}
if $0 \leq k_i \leq n_i$ for all $i \in I$ (we will use the notation $\b0 < \bk < \bn$ if we wish to further indicate that $\bk \neq \b0$ and $\bk \neq \bn$). Finally, let:
\begin{equation}
\label{eqn:length}
|\bn| = \sum_{i\in I} n_i
\end{equation}

\subsection{}

The following is the key notion of the current Section. \\

\begin{definition}
\label{def:slope}

Let $\bm \in \qq$. We will say that $F \in \CA^+$ has \textbf{slope} $\leq \bm$ if:
\begin{equation}
\label{eqn:slope plus} 
\lim_{\xi \rightarrow \infty} \frac {F(\dots, \xi z_{i1},\dots, \xi z_{ik_i}, z_{i,k_i+1},\dots, z_{in_i},\dots)}{\xi^{\bm \cdot \bk + \langle \bk,\bn - \bk \rangle}}
\end{equation}
is finite for all $\b0 \leq \bk \leq \bn$. Similarly, we will say that $G \in \CA^-$ has \textbf{slope} $\leq \bm$ if:
\begin{equation}
\label{eqn:slope minus} 
\lim_{\xi \rightarrow 0} \frac {G(\dots, \xi z_{i1},\dots, \xi z_{ik_i}, z_{i,k_i+1},\dots, z_{in_i},\dots)}{\xi^{- \bm \cdot \bk - \langle \bn - \bk, \bk \rangle}} 
\end{equation}
is finite for all $\b0 \leq \bk \leq \bn$. \\

\end{definition}

\noindent We will also say that $F \in \CA^+$, respectively $G \in \CA^-$, has \textbf{naive slope} $\leq \bm$ if:
\begin{align}
&\vdeg F \leq \bm \cdot \hdeg F \label{eqn:naive slope plus} \\
&\vdeg G \geq \bm \cdot \hdeg G \label{eqn:naive slope minus}
\end{align}
The $\bk = \bn$ case of formulas \eqref{eqn:slope plus}--\eqref{eqn:slope minus} shows that having slope $\leq \bm$ implies having naive slope $\leq \bm$. This fact can also be seen as a particular case of the following. \\

\begin{proposition}
\label{prop:cop slope}

An element $F \in \CA^+$ has slope $\leq \bm$ if and only if: 
\begin{equation}
\label{eqn:cop slope plus}
\Delta(F) = (\text{anything}) \otimes (\text{naive slope} \leq \bm) 
\end{equation}
Similarly, an element $G \in \CA^-$ has slope $\leq \bm$ if and only if: 
\begin{equation}
\label{eqn:cop slope minus}
\Delta(G) = (\text{naive slope} \leq \bm) \otimes (\text{anything})\end{equation}
The meaning of the RHS of \eqref{eqn:cop slope plus}/\eqref{eqn:cop slope minus} is that $\Delta(F)$/$\Delta(G)$ is an infinite sum of tensors, all of whose second/first factors have naive slope $\leq \bm$. Moreover, the statements above would remain true if we replaced ``naive slope" by ``slope". \\

\end{proposition}

\begin{proof} Let us prove the statements pertaining to $F$, and leave the analogous case of $G$ as an exercise to the interested reader. We will write:
$$
F = \sum_{\{d_{ia}\}^{i \in I}_{1\leq a \leq n_i}} \text{coefficient} \cdot \prod^{i\in I}_{1 \leq a \leq n_i} z_{ia}^{d_{ia}}
$$
for various coefficients. Note that:
$$
\zeta_{ji} \left(\frac 1x \right)^{-1} \in x^{\#_{\oji}} \BF[[x]]^\times
$$
Then as a consequence of \eqref{eqn:cop plus}, we have:
\begin{equation}
\label{eqn:cop proof}
\Delta(F) = \sum_{\b0 \leq \bk \leq \bn} \sum_{\{d_{ia}\}^{i \in I}_{1\leq a \leq n_i}} \mathop{\sum_{\{p_{jb} \geq 0\}^{j \in I}_{k_j < b \leq n_j}}}_{\{e^{ia}_{jb} \geq \#_{\oji}\}^{i,j \in I}_{1\leq a \leq k_i,  k_j < b \leq n_j}} \text{coefficient} \cdot
\end{equation}
$$
\prod^{j \in I}_{k_j < b \leq n_j} h_{j,+p_{jb}} \prod^{i\in I}_{1 \leq a \leq k_i} z_{ia}^{d_{ia}+\sum^{j \in I}_{k_j < b \leq n_j} e^{ia}_{jb}} \otimes \underbrace{\prod^{j\in I}_{k_j < b \leq n_j} z_{jb}^{d_{jb} - p_{jb} - \sum^{i \in I}_{1 \leq a \leq k_i} e^{ia}_{jb}}}_{\text{call this }F_2}
$$
The homogeneous degree of any second tensor factor in the formula above satisfies:
$$
\vdeg F_2 \leq \sum^{j \in I}_{k_j < b \leq n_j} \left( d_{jb} - \sum^{i \in I}_{1 \leq a \leq k_i} \#_{\oji} \right) = \sum^{j \in I}_{k_j < b \leq n_j} d_{jb} - \left \langle \bn - \bk, \bk \right \rangle
$$
By the assumption \eqref{eqn:slope plus}, the right-hand side of the expression above is:
$$
\leq \bm \cdot (\bn - \bk) = \bm \cdot \hdeg F_2
$$
which implies that $F_2$ has naive slope $\leq \bm$. Conversely, the terms $F_2$ in \eqref{eqn:cop proof} with maximal naive slope are the ones corresponding to $p_{jb} = 0$ and $e^{ia}_{jb} = \#_{\oji}$. If these $F_2$'s have naive slope $\leq \bm$, then the chain of inequalities above implies precisely:
$$
\sum^{j \in I}_{k_j < b \leq n_j} d_{jb} \leq \bm \cdot (\bn-\bk) +\left \langle \bn - \bk, \bk \right \rangle
$$
Since this holds for all $\b0 \leq \bk \leq \bn$, we obtain precisely \eqref{eqn:slope plus}. \\

\noindent It remains to show that we can replace the weaker notion of ``naive slope" by the stronger ``slope" in \eqref{eqn:cop slope plus}. To this end, we will explicitly show that if we write:
$$
\Delta(F) = \sum_{s \in S} F_{1,s} \otimes F_{2,s}
$$
where $\{F_{1,s}\}_{s\in S}$ is an arbitrary linear basis of $\CA^{\geq}$, then every $F_{2,s}$ has slope $\leq \bm$. The key to proving this fact is the coassociativity of the coproduct:
$$
(\text{Id} \otimes \Delta) \circ \Delta(F) = (\Delta \otimes \text{Id}) \circ \Delta(F)
$$
The left-hand side of the expression above is precisely:
$$
\sum_{s \in S} F_{1,s} \otimes \Delta(F_{2,s})
$$
while the right-hand side is of the form:
$$
(\text{anything}) \otimes (\text{anything}) \otimes (\text{naive slope} \leq \bm)
$$
by \eqref{eqn:cop slope plus}. For any given $s \in S$, identifying the coefficients of $F_{1,s} \otimes - \otimes -$ in the two expressions above implies that $\Delta(F_{2,s}) = (\text{anything}) \otimes (\text{naive slope} \leq \bm)$. By \eqref{eqn:cop slope plus}, this precisely means that $F_{2,s}$ has slope $\leq \bm$, as we needed to prove. 

\end{proof} 

\noindent Let us denote the subspaces of shuffle elements of slope $\leq \bm$ by:
\begin{equation}
\label{eqn:slope subalgebras}
\CA^\pm_{\leq \bm} \subset \CA^\pm
\end{equation}
Proposition \ref{prop:cop slope} and the multiplicativity of $\Delta$ show that $\CA^\pm_{\leq \bm}$ are algebras.  \\

\subsection{} It is easy to see that the graded pieces of $\CA^\pm_{\leq \bm}$, namely:
$$
\CA_{\leq \bm|\pm \bn, \pm d} = \CA_{\pm \bn, \pm d} \cap \CA^\pm_{\leq \bm} 
$$
are finite-dimensional for any $(\bn, d) \in \nn \times \BZ$. This is because \eqref{eqn:slope plus}/\eqref{eqn:slope minus} impose upper/lower bounds on the exponents of the variables that make up the Laurent polynomials $F$/$G$. If we also fix the total homogeneous degree of such a polynomial, then there are finitely many choices for the monomials which make up $F$/$G$. \\

\begin{definition}
\label{def:subalg}

For any $\bm \in \qq$, we will write:
\begin{equation}
\label{eqn:sub b}
\CB^\pm_{\bm} \subset \CA^\pm
\end{equation}
for the subalgebras consisting of elements of slope $\leq \bm$ and naive slope $= \bm$ \footnote{Having naive slope $ = \bm$ means having equality in \eqref{eqn:naive slope plus}--\eqref{eqn:naive slope minus}. This terminology is slightly ambiguous, as there may be infinitely many values of $\bm$ for which equality holds, but in what follows this ambiguity will be clarified by the context}. \\

\end{definition}

\noindent We will denote the graded pieces of $\CB^\pm_{\bm}$ by:
\begin{equation}
\label{eqn:b graded pieces}
\CB^\pm_{\bm} = \bigoplus_{\bn \in \nn} \CB_{\bm|\pm \bn}
\end{equation}
where $\CB_{\bm|\pm \bn} = \CA_{\leq \bm|\pm \bn, \pm \bm \cdot \bn}$. If $\bm \cdot \bn \notin \BZ$ for some $\bn \in \nn$, the respective direct summand in \eqref{eqn:b graded pieces} is zero. As for the non-zero direct summands, they are all finite-dimensional, as was explained in the beginning of the current Subsection. \\

\subsection{} 
\label{sub:delta m}

We can make the algebras $\CB_{\bm}^\pm$ into Hopf algebras if we first extend them:
\begin{align}
&\CB_{\bm}^{\geq} = \CB^+_{\bm} \otimes \BF \left[ h^{\pm 1}_{i,+0}  \right]_{i \in I} \Big/ \text{relation \eqref{eqn:rel ext plus}} \label{eqn:ext plus} \\
&\CB_{\bm}^{\leq} = \CB^-_{\bm} \otimes \BF \left[ h^{\pm 1}_{i,-0}  \right]_{i \in I} \Big/ \text{relation \eqref{eqn:rel ext minus}} \label{eqn:ext minus}
\end{align}
There is a coproduct $\Delta_{\bm}$ on the subalgebras \eqref{eqn:ext plus}--\eqref{eqn:ext minus}, determined by:
$$
\Delta_{\bm}(h_{i,\pm 0}) = h_{i, \pm 0} \otimes h_{i,\pm 0}
$$
and the following formulas for any $F \in \CB_{\bm|\bn}$ and $G \in \CB_{\bm|-\bn}$:
\begin{align}
&\Delta_{\bm}(F) = \sum_{\b0 \leq \bk \leq \bn} \lim_{\xi \rightarrow \infty} \frac {h_{\bn-\bk} F(\dots, z_{i1},\dots, z_{ik_i} \otimes \xi z_{i,k_i+1},\dots, \xi z_{in_i},\dots)}{\xi^{\bm \cdot (\bn - \bk)} \cdot \text{lead} \left[ \prod^{i \in I}_{1\leq a \leq k_i} \prod^{j \in I}_{k_j < b \leq n_j} \zeta_{ji} \left( \frac {\xi z_{jb}}{z_{ia}} \right) \right]} \label{eqn:cop b plus} \\
&\Delta_{\bm}(G) = \sum_{\b0 \leq \bk \leq \bn} \lim_{\xi \rightarrow 0} \frac {G(\dots, \xi z_{i1},\dots, \xi z_{ik_i} \otimes z_{i,k_i+1},\dots, z_{in_i},\dots)h_{-\bk}}{\xi^{- \bm \cdot \bk} \cdot \text{lead} \left[ \prod^{i \in I}_{1\leq a \leq k_i} \prod^{j \in I}_{k_j < b \leq n_j} \zeta_{ij} \left( \frac {\xi z_{ia}}{z_{jb}} \right) \right]} \label{eqn:cop b minus}
\end{align}
where ``lead$[\dots]$" refers to the leading order term in $\xi$ of the expression marked by the ellipsis (expanded as $\xi \rightarrow \infty$ or as $\xi \rightarrow 0$, depending on the situation), and:
\begin{equation}
\label{eqn:cartan}
h_{\pm \bn} = \prod_{i\in I} h_{i,\pm 0}^{n_i}
\end{equation}
for all $\bn \in \nn$. By its very definition, $\Delta_{\bm}$ consists of the leading naive slope terms in formulas \eqref{eqn:cop plus}--\eqref{eqn:cop minus}, in the sense that:
\begin{align*}
&\Delta_{\bm}(F) = \text{component of } \Delta(F) \text{ in }  \bigoplus_{\bn = \bn_1 + \bn_2} h_{\bn_2} \CA_{\bn_1, \bm \cdot \bn_1} \otimes \CA_{\bn_2, \bm \cdot \bn_2} \\
&\Delta_{\bm}(G) = \text{component of } \Delta(G) \text{ in }  \bigoplus_{\bn = \bn_1 + \bn_2} \CA_{- \bn_1, - \bm \cdot \bn_1} \otimes \CA_{-\bn_2, -\bm \cdot \bn_2}  h_{- \bn_1} 
\end{align*}
for all $F \in \CB_{\bm}^{+}$, $G \in \CB_{\bm}^{-}$. Thus, the fact that $\Delta_{\bm}$ makes $\CB^{\geq}_{\bm}$ and $\CB^{\leq}_{\bm}$ into bialgebras is induced by the fact that $\Delta$ makes $\CA^{\geq}$ and $\CA^{\leq}$ into bialgebras. \\

\begin{proposition}
\label{prop:compatible 1}

The restriction of the pairing \eqref{eqn:hopf pairing} to:
\begin{equation}
\label{eqn:restricted pairing}
\langle \cdot, \cdot \rangle : \CB_{\bm}^{\geq} \otimes \CB_{\bm}^{\leq} \longrightarrow \BF
\end{equation}
satisfies properties \eqref{eqn:bialg 1}--\eqref{eqn:bialg 2} with respect to the coproduct $\Delta_{\bm}$. \\

\end{proposition}

\begin{proof} Let us check \eqref{eqn:bialg 1} and leave the analogous formula \eqref{eqn:bialg 2} as an exercise to the interested reader. Moreover, we will only consider the case when $a \in \CB_{\bm}^+$ and $b_1, b_2 \in \CB_{\bm}^-$, as the situation when one or more of $a,b_1,b_2$ is of the form \eqref{eqn:cartan} is quite easy, and so left as an exercise to the interested reader. Thus, let us write: 
$$
\Delta(a) = \sum_{s \in S} a_{1,s} \otimes a_{2,s}
$$
where $S$ is some indexing set. Formula \eqref{eqn:cop slope plus} gives us:
\begin{equation}
\label{eqn:equivalent}
\vdeg a_{2,s} \leq \bm \cdot (\hdeg a_{2,s}) \quad \Leftrightarrow \quad \vdeg a_{1,s} \geq \bm \cdot (\hdeg a_{1,s}) 
\end{equation}
(the equivalence is due to the fact that $\vdeg a =  \bm \cdot (\hdeg a)$, on account of the very definition of $\CB_{\bm}^+ \ni a$). The definition of $\Delta_{\bm}$ implies that:
$$
\Delta_{\bm}(a) = \sum_{s \in S'} a_{1,s} \otimes a_{2,s}
$$
where the indexing set $S' \subset S$ consists of those $s \in S$ for which equality holds in \eqref{eqn:equivalent}. The fact that \eqref{eqn:bialg 1} holds with respect to $\Delta$ means that:
$$
 \langle a, b_1b_2 \rangle = \sum_{s\in S} \left \langle a_{1,s}, b_1 \right \rangle \left \langle a_{2,s}, b_2 \right \rangle
$$
However, because $\vdeg b_{1,2} = \bm \cdot (\hdeg b_{1,2})$, the pairings in the formula above are non-zero only if we have equality in \eqref{eqn:equivalent}, i.e. only if $s \in S'$. Therefore:
$$
 \langle a, b_1b_2 \rangle = \sum_{s\in S'} \left \langle a_{1,s}, b_1 \right \rangle \left \langle a_{2,s}, b_2 \right \rangle
$$
which precisely states that \eqref{eqn:bialg 1} also holds with respect to the coproduct $\Delta_{\bm}$. 

\end{proof}

\subsection{}
\label{sub:double b}

The pairing \eqref{eqn:restricted pairing} is non-degenerate \footnote{As usual in the theory of quantum groups, this statement is true as stated for the restricted pairing \eqref{eqn:pairing sub} to the $\pm$ subalgebras. To have the statement hold for the $\geq, \leq$ subalgebras, one needs to work instead over the power series ring in $\log(q), \log(t_e)$ instead of over $\BF = \BQ(q,t_e)_{e \in E}$}, as we will show in Proposition \ref{prop:non deg}. This will allow us to define the Drinfeld double:
\begin{equation}
\label{eqn:double b}
\CB_{\bm} = \CB_{\bm}^\geq \otimes \CB_{\bm}^{\leq} 
\end{equation}
as in Subsection \ref{sub:double}, which has a universal $R$-matrix as in Subsection \ref{sub:r-matrix}:
$$
\CR_{\bm} \in \CB^{\geq}_{\bm} \ \widehat{\otimes} \ \CB^{\leq}_{\bm} \subset \CB_{\bm} \ \widehat{\otimes} \ \CB_{\bm}
$$
Explicitly, $\CR_{\bm}$ is the canonical tensor of the pairing \eqref{eqn:restricted pairing}. As in \eqref{eqn:univ r}, we have:
\begin{equation}
\label{eqn:univ r sub}
\CR_{\bm} = \CR'_{\bm} \cdot \Big[\text{a sum of products involving the }h_{i,\pm 0} \Big]
\end{equation}
where $\CR'_{\bm}$ is the canonical tensor of the restriction of the pairing \eqref{eqn:restricted pairing} to:
\begin{equation}
\label{eqn:pairing sub}
\langle \cdot , \cdot \rangle : \CB^+_{\bm} \otimes \CB^-_{\bm} \longrightarrow \BF
\end{equation}
Although they look similar, we emphasize the fact that the Drinfeld doubles $\CA$ and $\CB_{\bm}$ are defined with respect to the different coproducts $\Delta$ and $\Delta_{\bm}$, respectively. Since the product in a Drinfeld double (namely relation \eqref{eqn:double formula}) is controlled by the coproduct that is used to define the double, the following result is non-trivial. \\

\begin{proposition}
\label{prop:compatible 2}

The inclusion map $\CB_{\bm} \subset \CA$ (obtained by tensoring together the natural inclusion maps $\CB_{\bm}^\geq \subset \CA^\geq$ and $\CB_{\bm}^\leq \subset \CA^\leq$) is an algebra homomorphism. \\

\end{proposition}

\begin{proof} Consider any $a \in \CB^{\geq}_{\bm}$ and $b \in \CB^{\leq}_{\bm}$, and let us write:
$$
\Delta(a) = \sum_{s \in S} a_{1,s} \otimes a_{2,s} \qquad \text{and} \qquad \Delta(b) = \sum_{t \in T} b_{1,t} \otimes b_{2,t}
$$
for some sets $S$ and $T$. By the definition of $\Delta_{\bm}$, we have:
$$
\Delta_{\bm}(a) = \sum_{s \in S'} a_{1,s} \otimes a_{2,s} \quad \text{ and } \quad \Delta_{\bm}(b) = \sum_{t \in T'} b_{1,t} \otimes b_{2,t}
$$
where the indexing sets $S' \subset S$, $T' \subset T$ consists of those $s \in S$, $t \in T$ such that:
\begin{align*}
&\vdeg a_{1,s} = \bm \cdot (\hdeg a_{1,s}) \quad \Leftrightarrow \quad \vdeg a_{2,s} = \bm \cdot (\hdeg a_{2,s}) \\ 
&\vdeg b_{1,t} = \bm \cdot (\hdeg b_{1,t}) \ \quad \Leftrightarrow \quad \vdeg b_{2,t} = \bm \cdot (\hdeg b_{2,t})
\end{align*}
Formula \eqref{eqn:double formula} implies that the following relation holds in $\CA$:
\begin{equation}
\label{eqn:formula}
\sum_{s \in S, t \in T} a_{1,s} b_{1,t} \left\langle a_{2,s},b_{2,t} \right \rangle = \sum_{s \in S, t \in T} b_{2,t} a_{2,s} \left \langle a_{1,s},b_{1,t} \right \rangle 
\end{equation}
However, formulas \eqref{eqn:cop slope plus}--\eqref{eqn:cop slope minus} imply that $a_{2,s}$ and $b_{1,t}$ have naive slope $\leq \bm$, for all $s \in S$ and $t \in T$. This implies that:
\begin{align*}
&\vdeg a_{2,s} \leq \bm \cdot (\hdeg a_{2,s}) \quad \Rightarrow \quad \vdeg a_{1,s} \geq \bm \cdot (\hdeg a_{1,s}) \\ 
&\vdeg b_{1,t} \geq \bm \cdot (\hdeg b_{1,t}) \ \quad \Rightarrow \quad \vdeg b_{2,t} \leq \bm \cdot (\hdeg b_{2,t})
\end{align*}
where in both cases, the implication is due to our assumption that $\vdeg a = \bm \cdot (\hdeg a)$ and $\vdeg b = \bm \cdot (\hdeg b)$. Therefore, the only way for the pairings in the left/right hand sides of \eqref{eqn:formula} to be non-zero is to have equality in all the inequalities above, which would imply $s \in S'$ and $t \in T'$. We therefore have:
$$
\sum_{s \in S', t \in T'} a_{1,s} b_{1,t} \left\langle a_{2,s}, b_{2,t} \right \rangle = \sum_{s \in S', t \in T'} b_{2,t} a_{2,s} \left \langle a_{1,s}, b_{1,t} \right \rangle 
$$
However, the equality above is simply \eqref{eqn:double formula} in the double $\CB_{\bm}$, which implies that the same multiplicative relations hold in $\CA$ as in $\CB_{\bm}$. 

\end{proof}

\subsection{} 

Let us now fix $\bm \in \qq$ and $\bth \in \qqp$, and consider the subalgebras $\{\CB_{\bm+r\bth}\}_{r\in \BQ}$. \\

\begin{proposition}
\label{prop:pbw pairing}

For any $\bm \in \qq$ and $\bth \in \qqp$, we have:
\begin{equation}
\label{eqn:pbw pairing}
\left \langle \prod_{r \in \BQ}^{\rightarrow} a_r, \prod_{r \in \BQ}^{\rightarrow} b_r \right \rangle = \prod_{r \in \BQ}^{\rightarrow} \langle a_r, b_r \rangle
\end{equation}
for all elements $\left\{a_r \in \CB_{\bm + r \bth}^+, b_r \in \CB_{\bm + r \bth}^-\right\}_{r\in \BQ}$, almost all of which are equal to 1. \\

\end{proposition}

\begin{proof} Let $r \in \BQ$ be maximal such that $a_r \neq 1$ or $b_r \neq 1$, and let us assume that $|\hdeg a_r| \geq - |\hdeg b_r|$ (the opposite case is treated analogously, so we leave it as an exercise to the interested reader). Then formula \eqref{eqn:bialg 2} implies:
\begin{equation}
\label{eqn:pair many 1}
\left \langle \prod_{r' < r} a_{r'} \cdot a_r, \prod_{r' \leq r} b_{r'} \right \rangle = \left \langle \prod_{r' < r} a_{r'}, \prod_{r' \leq r} b_{r',2} \right \rangle\left \langle a_r, \prod_{r' \leq r} b_{r',1} \right \rangle
\end{equation}
where we use Sweedler notation $\Delta(b_r) = b_{r,1} \otimes b_{r,2}$. Because of \eqref{eqn:cop slope minus}, all of $b_{r',1}$ with $r' < r$ have slope strictly smaller than $\bm + r \bth$. Therefore, the only term in the right-hand side of the expression above which could pair non-trivially with $a_r \in \CB_{\bm+r\bth}^+$ is $b_{r,1}$. However, because of the assumption that $|\hdeg a_r| \geq - |\hdeg b_r|$, such a non-trivial pairing is only possible if all three of the properties below hold: \\

\begin{itemize}[leftmargin=*]

\item $\hdeg a_r = - \hdeg b_r$ \\

\item $b_{r,1}$ is the first tensor factor in the first summand below:
$$
\Delta(b_r) = b_r \otimes h_{-\hdeg b_r} + \dots
$$

\item $b_{r',1}$ is the first tensor factor in the first summand below:
$$
\Delta(b_{r'}) = 1 \otimes b_{r'} + \dots
$$
for all $r' < r$. \\

\end{itemize}

\noindent Therefore, \eqref{eqn:pair many 1} implies:
\begin{equation}
\label{eqn:pair many 2}
\left \langle \prod_{r' \leq r} a_{r'}, \prod_{r' \leq r} b_{r'} \right \rangle = \left \langle \prod_{r' < r} a_{r'}, \prod_{r' < r} b_{r'} \cdot h_{-\hdeg b_r} \right \rangle\left \langle a_r,b_r \right \rangle
\end{equation}
Since \eqref{eqn:bialg 1} implies the following identity for any $a \in \CA^+$, $b \in \CA^-$ and any $\bn \in \nn$:
$$
\left \langle a, b \cdot h_{-\bn} \right \rangle = \langle a,b \rangle  \left \langle 1, h_{-\bn} \right \rangle = \langle a,b \rangle
$$
the right-hand side of \eqref{eqn:pair many 2} is unchanged if we remove $h_{-\hdeg b_r}$. Iterating identity \eqref{eqn:pair many 2} for the various $r'$ for which $a_{r'} \neq 1$ or $b_{r'} \neq 1$, one obtains identity \eqref{eqn:pbw pairing}. 

\end{proof}

\subsection{}

Still fixing $\bm \in \qq$ and $\bth \in \qqp$ as before, our main goal in the following Subsections (see Corollary \ref{cor:iso}) is to prove that multiplication yields isomorphisms:
\begin{equation}
\label{eqn:pbw}
\bigotimes_{r \in \BQ}^{\rightarrow} \CB_{\bm+r\bth}^\pm \xrightarrow{\sim} \CA^\pm
\end{equation}
If we write $\CB_{\bm + \infty \bth}^{\pm} = \BF[h_{i,\pm 0}^{\pm 1}, h_{i,\pm 1}, \dots]_{i\in I}$, then \eqref{eqn:pbw} leads to isomorphisms:
\begin{equation}
\label{eqn:pbw double}
\bigotimes_{r \in \BQ \sqcup \infty}^{\rightarrow} \CB^+_{\bm+r\bth} \xrightarrow{\sim} \CA^{\geq} \qquad \text{and} \qquad \bigotimes_{r \in \BQ \sqcup \infty}^{\rightarrow} \CB^-_{\bm+r\bth} \xrightarrow{\sim} \CA^{\leq}
\end{equation}
Thus, the entire $\CA = \CA^{\geq} \otimes \CA^{\leq}$ factors as the tensor product of the $\{\CB_{\bm+r\bth}^\pm\}_{r\in \BQ \sqcup \infty}$. \\

\begin{proposition}
\label{prop:surj}

For any $\bm \in \qq$, $\bth \in \qqp$ and $p \in \BQ$, the multiplication map:
\begin{equation}
\label{eqn:pbw finite}
\bigotimes_{r \in \BQ_{\leq p}}^{\rightarrow} \CB_{\bm+r\bth}^\pm \twoheadrightarrow \CA^\pm_{\leq \bm+p\bth}
\end{equation}
is surjective (above, $\BQ_{\leq p}$ denotes the set of rational numbers $\leq p$). \\

\end{proposition}

\noindent Since any element of $\CA^+$ has slope $\leq \bm+r\bth$ for $r \in \BQ$ large enough (this is because $\bth \in \qqp$), the surjectivity of \eqref{eqn:pbw finite} implies that the multiplication map:
\begin{equation}
\label{eqn:pbw infinite}
\bigotimes_{r \in \BQ}^{\rightarrow} \CB_{\bm+r\bth}^\pm \twoheadrightarrow  \CA^\pm
\end{equation}
is also surjective. \\

\begin{proof} \emph{of Proposition \ref{prop:surj}:} Let us consider the restriction of the multiplication map \eqref{eqn:pbw finite} to the subspaces of given degree $(\bn,d) \in \nn \times \BZ$:
\begin{equation}
\label{eqn:pbw deg}
\mathop{\bigoplus_{\sum_{r \in \BQ_{\leq p}} \bn_r = \bn}}_{\sum_{r\in\BQ_{\leq p}} (\bm+r\bth)\cdot \bn_r = d} \left[ \bigotimes_{r \in \BQ_{\leq p}}^{\rightarrow} \CB_{\bm+r\bth|\pm \bn_r} \right] \xrightarrow{\phi_{\pm \bn, \pm d}} \CA_{\leq \bm+p\bth|\pm \bn,\pm d}
\end{equation}
(the indexing set goes over all sequences $(\bn_r)_{r \in \BQ_{\leq p}}$ of elements of $\nn$, almost all of which are 0). We will prove that $\phi_{\pm \bn, \pm d}$ is surjective by induction on $\bn$, with respect to the ordering \eqref{eqn:ineq deg} (the base case, when $\bn = \bs^i$ for some $i \in I$, is trivial). To streamline the subsequent explanation, if a certain element of the shuffle algebra has slope (or naive slope) $\bm + r\bth$, we will actually refer to the number $r \in \BQ$ as its slope (or naive slope). This also has the added benefit of making the notion ``naive slope $=r$" unambiguous, as the fact that $\bth \in \qqp$ means that for any $(\bn,d) \in \nn \times \BZ$, there exists exactly one rational number $r$ for which $(\bm+r\bth)\cdot \bn = d$. \\

\noindent So let us show that any element $F \in \CA_{\leq \bm + p\bth|\bn, d}$ lies in the image of the map $\phi_{\bn,d}$ (we will only discuss the case $\pm = +$, as the $\pm = -$ case is analogous). Let $r \leq p$ denote the naive slope of $F$, and let us call \textbf{hinges} those:
\begin{equation}
\label{eqn:hinge}
(\bk,e) \in \nn \times \BZ
\end{equation}
such that $\Delta(F)$ has a non-zero component in:
\begin{equation}
\label{eqn:deg hinge}
\CA_{\bn - \bk, d-e} \otimes \CA_{\bk,e}
\end{equation}
Clearly, a hinge would need to satisfy \eqref{eqn:ineq deg}, and by \eqref{eqn:cop slope plus}, also the inequality:
\begin{equation}
\label{eqn:ineq 1}
e \leq (\bm+p\bth) \cdot \bk
\end{equation}
We will call a hinge \textbf{bad} if:
\begin{equation}
\label{eqn:ineq 2}
e > (\bm + r \bth) \cdot \bk
\end{equation}
It is easy to see that $F$ has finitely many bad hinges, as there are only finitely many values of $\bk$ satisfying \eqref{eqn:ineq deg}, and for any such $\bk$, finitely many integers $e$ that satisfy the inequalities \eqref{eqn:ineq 1} and \eqref{eqn:ineq 2}. If $F$ has no bad hinges, then by \eqref{eqn:cop slope plus}, $F$ would lie in $\CB^+_{\bm + r \bth}$ and we would be done. Thus, our strategy will be to successively subtract from $F$ elements in the image of $\phi_{\bn,d}$ so as to ``kill" all its bad hinges. \\

\noindent The slope of a bad hinge \eqref{eqn:hinge} is that rational number $\rho \in  (r,p]$ such that:
\begin{equation}
\label{eqn:slope of hinge}
e = (\bm+\rho \bth) \cdot \bk
\end{equation}
Let us consider the partial order on the set of bad hinges, given primarily by slope, and then by $|\bk|$ to break ties between hinges of the same slope. Let us write $(\bk,e)$ for a maximal bad hinge of $F$. The number $\rho$ from \eqref{eqn:slope of hinge} is minimal such that:
\begin{equation}
\label{eqn:slope of f}
F \in \CA^+_{\leq \bm + \rho \bth}
\end{equation}
By the maximality of $(\bk,e)$, the component of $\Delta(F)$ in degree \eqref{eqn:deg hinge} is given by:

\begin{multline}
\label{eqn:hinge cop 0}
\Delta_{(\bn-\bk,d-e), (\bk,e)}(F) = \\ = \text{top} \left[ \frac {h_{\bk} F(\dots, z_{i1},\dots, z_{i,n_i - k_i} \otimes \xi z_{i,n_i - k_i+1},\dots, \xi z_{in_i},\dots)}{\gamma \cdot \prod^{i \in I}_{1\leq a \leq n_i - k_i} \prod^{j \in I}_{n_j - k_j < b \leq n_j} \left( \frac {\xi z_{jb}}{z_{ia}} \right)^{\#_{\oji}}} \right]
\end{multline}
where ``top$[\dots]$" refers to the top coefficient in $\xi$ of the expression marked by the ellipsis, and $\gamma \in \BF^\times$. The reason for the formula above is that the maximality of $(\bk,e)$ implies that only the leading term of the $h$ power series in the numerator (respectively the $\zeta$ rational functions in the denominator) of \eqref{eqn:cop plus} can contribute (see \eqref{eqn:cop proof}). By writing $F$ as a linear combination of monomials, we have:
\begin{equation}
\label{eqn:hinge cop 1}
\Delta_{(\bn-\bk,d-e), (\bk,e)}(F) = \sum_{s \in S} h_{\bk} F_{1,s} \otimes F_{2,s}
\end{equation}
where $S$ is some indexing set, and $F_{1,s}, F_{2,s}$ denote various elements in $\CA^+$ that one obtains by summing up the various top coefficients in $\xi$ of \eqref{eqn:hinge cop 0}. \\


\begin{claim}
\label{claim:hinge}

The following element: 
\begin{equation}
\label{eqn:g}
G = \sum_{s \in S} F_{1,s} F_{2,s}
\end{equation}
lies in the image of $\phi_{\bn,d}$. All its bad hinges are less than or equal to $(\bk,e)$, and:
\begin{equation}
\label{eqn:gamma}
\Delta_{(\bn-\bk,d-e), (\bk,e)}(F) = \gamma' \cdot \Delta_{(\bn-\bk,d-e), (\bk,e)}(G)
\end{equation}
for some $\gamma' \in \BF^\times$. \\

\end{claim}

\noindent Let us complete the induction step of the surjectivity of $\phi_{\bn,d}$. Claim \ref{claim:hinge} allows us to reduce the fact that $F$ lies in the image of $\phi_{\bn,d}$ to the analogous fact for $F - \gamma' G$, where $\gamma'$ is the constant that features in \eqref{eqn:gamma}. Moreover, Claim \ref{claim:hinge} implies that $F-\gamma' G$ does not in fact have a bad hinge at $(\bk,e)$. Thus, repeating this argument finitely many times allows us to reduce $F$ to an element without any bad hinges, which as we have seen must lie in $\CB_{\bm+r\bth|\bn}$. This concludes the induction step. \\

\begin{proof} \emph{of Claim \ref{claim:hinge}:} To eliminate redundancy in the sum \eqref{eqn:hinge cop 1}, we will assume the various $F_{1,s}$ which appear are part of a fixed linear basis of $\CA_{\bn - \bk, d - e}$. Then \eqref{eqn:slope of f} together with the last sentence of Proposition \ref{prop:cop slope} imply that $F_{2,s}$ has slope $\leq \rho$ for all $s \in S$. Because $F_{2,s}$ has naive slope $ = \rho$ by \eqref{eqn:slope of hinge}, we conclude that: 
\begin{equation}
\label{eqn:f2}
F_{2,s} \in \CB_{\bm+\rho\bth|\bk}
\end{equation}
for all $s \in S$. Let us now express every $F_{2,s}$ in terms of a fixed linear basis:
$$
\{F_{2,t}\}_{t \in T} \quad \text{of} \quad \CB_{\bm+\rho\bth|\bk}
$$
and then re-express \eqref{eqn:hinge cop 1} in this new basis:
\begin{equation}
\label{eqn:hinge cop 2}
\Delta_{(\bn-\bk,d-e), (\bk,e)}(F) = \sum_{t \in T} h_{\bk} F_{1,t} \otimes F_{2,t}
\end{equation}

\begin{claim}
\label{claim:final}

Every $F_{1,t}$ which appears in \eqref{eqn:hinge cop 2} has slope $<\rho$. \\

\end{claim}

\noindent By Claim \ref{claim:final} and the induction hypothesis of the surjectitivity of the map \eqref{eqn:pbw deg}:
$$
G = \sum_{t \in T} F_{1,t} F_{2,t}
$$
is a sum of products of elements of $\{\CB_{\bm+r\bth}\}_{r \leq \rho}$ in increasing order of $r$. This implies that $G \in \text{Im }\phi_{\bn,d}$. To compute the bad hinges of $G$, we note that:
$$
\Delta(G) = \sum_{t \in T} \underbrace{\Delta(F_{1,t})}_{X_1 \otimes X_2} \underbrace{\Delta(F_{2,t})}_{Y_1 \otimes Y_2}
$$
Since every $F_{1,t}$ has slope $<\rho$ and every $F_{2,t}$ has slope $=\rho$, then every $X_2$ that appears in the formula above has naive slope $<\rho$ and every $Y_2$ has naive slope $\leq \rho$. But unless $X_2 = 1$ and $Y_2 = F_{2,t}$, the product $X_2Y_2$ either has naive slope $<\rho$, or it has naive slope $= \rho$ but smaller $|\text{hdeg}|$ than $|\bk|$, and thus cannot contribute to the component of $\Delta(G)$ in degree \eqref{eqn:deg hinge}. Thus, we have:
$$
\Delta_{(\bn-\bk,d-e),(\bk,e)}(G) = \sum_{t\in T} (F_{1,t} \otimes 1) \left( h_{\bk} \otimes F_{2,t} \right)
$$
This matches \eqref{eqn:hinge cop 1} up to an overal constant that one obtains when commuting $h_{\bk}$ past the various $F_{1,t}$ (this constant can be read off from \eqref{eqn:rel ext plus}, and it only depends on $\bk$ and the horizontal degree of the $F_{1,t}$'s, which is equal to $\bn-\bk$ for all $t \in T$). \\


\begin{proof} \emph{of Claim \ref{claim:final}:} By \eqref{eqn:slope plus}, the fact that $F$ has slope $\leq \rho$, see \eqref{eqn:slope of f}, we have:
\begin{equation}
\label{eqn:inequality}
\text{total degree of }F \text{ in }\{z_{i1},\dots,z_{il_i}\}_{i\in I}  \leq (\bm+\rho \bth)\cdot \bl + \langle \bl, \bn-\bl \rangle
\end{equation}
for all $\b0 \leq \bl \leq \bn$. The inequality is strict if $|\bk| < |\bl|$, on account of the maximality of the hinge $(\bk,e)$. By the symmetry of the Laurent polynomial $F$, the same inequality holds if we replace the set of variables $z_{i1},\dots,z_{il_i}$ by any other subset of $l_i$ of the variables $z_{i1},\dots,z_{in_i}$. Let us zoom in on a certain monomial $\mu$ that appears in the Laurent polynomial $F$, and write for all $\b0 < \bk' \leq \bn - \bk$:
\begin{align*}
&\alpha = \text{total degree of }\mu \text{ in } \{z_{i1},\dots,z_{ik'_i}\}_{i \in I} \\
&\beta = \text{total degree of }\mu \text{ in } \{z_{i,n_i-k_i+1},\dots,z_{in_i}\}_{i \in I} 
\end{align*}
By applying \eqref{eqn:inequality} for $\bl= \bk+\bk'$, we conclude that:
\begin{equation}
\label{eqn:inequality 1}
\alpha + \beta < (\bm+\rho \bth)\cdot (\bk+\bk') + \langle \bk+\bk', \bn-\bk-\bk' \rangle
\end{equation}
On the other hand, if the monomial $\mu$ survives in the limit \eqref{eqn:hinge cop 0}, this implies:
\begin{equation}
\label{eqn:inequality 2}
\beta = (\bm+\rho \bth)\cdot \bk + \langle \bk, \bn-\bk \rangle
\end{equation}
Subtracting \eqref{eqn:inequality 2} from \eqref{eqn:inequality 1} yields:
\begin{equation}
\label{eqn:inequality 3}
\alpha < (\bm+\rho \bth)\cdot \bk' + \langle \bk', \bn-\bk-\bk' \rangle - \langle \bk, \bk' \rangle
\end{equation}
However, the homogeneous degree of the first tensor factor of \eqref{eqn:hinge cop 0} in the variables $\{z_{i1},\dots,z_{ik'_i}\}_{i\in I}$ is equal to $\alpha + \langle \bk, \bk' \rangle$. By inequality \eqref{eqn:inequality 3}, this quantity is: 
$$
< (\bm+\rho \bth)\cdot \bk' + \langle \bk', \bn-\bk-\bk' \rangle
$$
According to $\eqref{eqn:slope plus}$, this precisely means that $F_{1,t}$ has slope $< \rho$ for all $t \in T$. 

\end{proof}

\end{proof}

\end{proof}

\subsection{}

We are now ready to prove that the pairing \eqref{eqn:pairing sub} is non-degenerate, which is a necessary hypothesis when constructing the Drinfeld double  \eqref{eqn:double b}. \\

\begin{proposition}
\label{prop:non deg}

For any $\bm \in \qq$, the pairing \eqref{eqn:pairing sub} is non-degenerate. \\

\end{proposition}

\begin{proof} Because the pairing \eqref{eqn:pairing positive} is non-degenerate, we have:
$$
\langle F, \CA^-\rangle = 0 \quad \Rightarrow \quad F = 0
$$
However, the surjectivity of the map \eqref{eqn:pbw infinite} allows us to write:
$$
\left \langle F, \bigotimes_{r \in \BQ}^{\rightarrow} \CB_{\bm+r\bth}^- \right \rangle = 0 \quad \Rightarrow \quad F = 0
$$
if $F \in \CB_{\bm}^+$, then Proposition \ref{prop:pbw pairing} implies that $F$ pairs trivially with all ordered products of elements from different $\CB_{\bm+r\bth}^-$'s, and only pairs non-trivially with $\CB_{\bm}^-$ itself. We conclude that:
$$
\left \langle F, \CB_{\bm}^- \right \rangle = 0 \quad \Rightarrow \quad F = 0
$$
which is precisely the non-degeneracy of \eqref{eqn:pairing sub} in the first factor. The case of non-degeneracy in the second factor is completely analogous.

\end{proof}

\begin{proposition}
\label{prop:inj}

For any $\bm \in \qq$ and $\bth\in \qqp$, the map \eqref{eqn:pbw infinite} is injective. \\

\end{proposition}

\noindent In particular, the Proposition implies that the maps \eqref{eqn:pbw finite} are injective for all $p \in \BQ$. \\

\begin{proof} By Proposition \ref{prop:non deg}, we may fix dual linear bases:
\begin{equation}
\label{eqn:dual bases}
\{ a_{r,s} \}_{s \in \BN} \subset \CB_{\bm+r\bth}^+ \qquad \text{and} \qquad \{ b_{r,s} \}_{s \in \BN} \subset \CB_{\bm+r\bth}^-
\end{equation}
for all $r \in \BQ$. We will assume $a_{r,0} = b_{r,0} = 1$ is an element in our bases. If the map \eqref{eqn:pbw infinite} (we assume $\pm = +$, as the case of $\pm = -$ is analogous) failed to be injective, then there would exist a non-trivial linear relation:
\begin{equation}
\label{eqn:non-trivial}
\sum_{\{s_r\}_{r \in \BQ}} \gamma_{\{s_r\}} \prod_{r \in \BQ}^{\rightarrow} a_{r,s_r} = 0 \in \CA^+
\end{equation}
(the sum goes over all collections of indices $s_r$, almost all of which are equal to 0). For any fixed collection of indices $\{t_r\}_{r \in \BQ}$, almost all of which are equal to 0, the relation above implies:
$$
\sum_{\{s_r\}_{r \in \BQ}} \gamma_{\{s_r\}} \left \langle \prod_{r \in \BQ}^{\rightarrow} a_{r,s_r}, \prod_{r \in \BQ}^{\rightarrow} b_{r,t_r} \right \rangle = 0 \in \CA^+
$$
By Proposition \ref{prop:pbw pairing}, the only pairing which survives in the formula above is the one for $s_r = t_r, \forall r \in \BQ$, thus implying that $\gamma_{\{t_r\}} = 0$. Since this holds for all collections of indices $\{t_r\}_{r \in \BQ}$, this precludes the existence of a non-trivial relation \eqref{eqn:non-trivial} and establishes the injectivity of the map \eqref{eqn:pbw infinite}. 

\end{proof}

\begin{corollary}
\label{cor:iso}

For any $\bm \in \qq$, $\bth\in \qqp$ and $p\in \BQ$, the maps \eqref{eqn:pbw finite} and \eqref{eqn:pbw infinite} are isomorphisms. \\

\end{corollary}

\noindent We have completed the construction of the isomorphisms \eqref{eqn:pbw}. As these isomorphisms preserve the pairing in the sense of Proposition \ref{prop:pbw pairing}, we have the following. \\

\begin{corollary}
\label{cor:r-matrix}

For any $\bm \in \qq$ and $\bth\in \qqp$, we have:
\begin{equation}
\label{eqn:factorization}
\CR' = \prod_{r \in \BQ}^{\rightarrow} \CR'_{\bm+r\bth}
\end{equation}
where $\CR'$ is defined in \eqref{eqn:univ r}, and $\CR'_{\bm}$ is defined in \eqref{eqn:univ r sub}. \\

\end{corollary}

\begin{proof} Let us consider dual bases \eqref{eqn:dual bases}. The canonical tensor of the pairing \eqref{eqn:pairing sub} (for $\bm$ replaced by $\bm+r\bth$) is:
\begin{equation}
\label{eqn:r1}
\CR'_{\bm+r\bth} = \sum_{s \in \BN} a_{r,s} \otimes b_{r,s}
\end{equation}
Meanwhile, by \eqref{eqn:pbw pairing} and \eqref{eqn:pbw}, we have that:
$$
\left\{ \prod_{r\in \BQ}^{\rightarrow} a_{r,s_r} \right\} \subset \CA^+ \qquad \text{and} \qquad \left\{ \prod_{r\in \BQ}^{\rightarrow} b_{r,s_r} \right\} \subset \CA^-
$$
(as $\{s_r\}_{r \in \BQ}$ goes over all collections of natural numbers, almost all of which are 0) are dual bases. Therefore, the canonical tensor of the pairing \eqref{eqn:pairing positive} is:
\begin{equation}
\label{eqn:r2}
\CR' = \sum_{\{s_r\}} \prod_{r\in \BQ}^{\rightarrow} a_{r,s_r}  \otimes \prod_{r\in \BQ}^{\rightarrow} b_{r,s_r} 
\end{equation}
Comparing formulas \eqref{eqn:r1} and \eqref{eqn:r2} yields \eqref{eqn:factorization}. 

\end{proof}

\subsection{} 
\label{sub:other coproducts}

The universal $R$-matrix intertwines the coproduct with its opposite:
$$
\Delta^{\op}(a) = \CR \cdot \Delta(a) \cdot \CR^{-1}
$$
for all $a \in \CA$. However, we now have the factorization:
$$
\CR = \prod_{r \in \BQ \sqcup \{\infty\}}^{\rightarrow} \CR'_{\bm+r\bth}
$$
where $\CR'_{\bm + \infty \bth}$ denotes the factor in square brackets in \eqref{eqn:univ r}. Therefore:
\begin{equation}
\label{eqn:other coproducts}
\Delta_{(\bm)}(a) = \left[\prod_{r \in \BQ_{>0} \sqcup \{\infty\}}^{\rightarrow} \CR'_{\bm+r\bth}\right] \cdot \Delta(a) \cdot \left[\prod_{r \in \BQ_{>0} \sqcup \{\infty\}}^{\rightarrow} \CR'_{\bm+r\bth}\right]^{-1}
\end{equation}
defines another coproduct on the algebra $\CA$, for all $\bm \in \qq$. We expect the restriction of $\Delta_{(\bm)}$ to the subalgebra $\CB_{\bm}$ to match the coproduct $\Delta_{\bm}$ of Subsection \ref{sub:delta m}. \\

\noindent The existence of many coproducts on quantum groups is a well-known phenomenon in representation theory. For example, when $Q$ is of finite type and $\CA$ is the corresponding quantum affine algebra, then $\Delta$ is the Drinfeld new coproduct and $\Delta_{(\b0)}$ is the Drinfeld-Jimbo coproduct. When $Q$ is the cyclic quiver and $\CA$ is the corresponding quantum toroidal algebra, we expect $\Delta_{(\b0)}$ to match the coproduct defined in \cite{Tale}. For general quivers, Conjecture \ref{conj:mo} suggests that the coproducts \eqref{eqn:other coproducts} match the ones defined by \cite{AO,MO,O1,O2,OS} using the theory of stable bases. \\

\begin{remark}
\label{rem:other tori}

All the results in the present paper would continue to hold if the equivariant parameters $\{q,t_e\}_{e \in E}$ were not generic, but specialized in any way which satisfies Assumption \begin{otherlanguage*}{russian}Ъ\end{otherlanguage*} of \cite{N}, for example:
$$
t_e = q^{\frac 12}, \ \forall e \in E
$$
Indeed, as explained in Section 5 of \loccitt, this assumption allows us to define $\CA$ as a Drinfeld double, and then all the notions of the current Section would carry through. The main caveat is that the wheel conditions \eqref{eqn:wheel} are no longer enough to define $\CS \subset \CV$; one needs to impose the stronger conditions \emph{(5.2)} of \loccit instead. \\

\end{remark}

\section{Connections to geometry}
\label{sec:geometry}

\medskip

\subsection{} We think of \eqref{eqn:pbw} as a PBW theorem for the shuffle algebras $\CA^\pm$: it says that a linear basis of $\CA^\pm$ is given by ordered products of linear bases of the subalgebras $\{\CB_{\bm+r\bth}\}_{r\in \BQ}$, for any fixed $\bm \in \qq$ and $\bth \in \qqp$. Moreover, by \eqref{eqn:pbw pairing}, these linear bases can be chosen to be dual to each other under the pairing \eqref{eqn:pairing positive}. Formula \eqref{eqn:factorization} also emphasizes the role the subalgebras $\CB^\pm_{\bm}$ play in understanding the universal $R$-matrix of $\CA$. This motivates our interest in understanding the subalgebras $\CB_{\bm}$. For starters, it is easy to see that the automorphisms \eqref{eqn:shift} send:
\begin{equation}
\label{eqn:shift sub}
\tau_{\bk} : \CB_{\bm} \xrightarrow{\sim} \CB_{\bm+\bk}
\end{equation}
for any $\bk \in \zz$. Therefore, the classification of the algebras $\CB_{\bm}$ only depends on $\bm \in (\BQ/\BZ)^I$. A more substantial reduction would be the following: \\

\begin{problem}
\label{prob:1}

Show that $\CB_{\bm}$ for the quiver $Q$ is isomorphic to the algebra $\CB_{\b0}$ for some other quiver $Q_{\bm}$, and understand the dependence of the latter on $\bm \in (\BQ/\BZ)^I$. \\

\end{problem}

\noindent For example, in \cite{Thesis}, when $Q$ is the cyclic quiver of length $n$, we showed that:
$$
\CB_{(m_1,\dots,m_n)} = U_q(\widehat{\mathfrak{gl}}_{n_1} ) \otimes \dots \otimes U_q(\widehat{\mathfrak{gl}}_{n_d})
$$
where the natural numbers $d$ and $n_1+\dots+n_d = n$ are defined by an explicit procedure from the rational numbers $m_1,\dots,m_n \in \BQ/\BZ$. In particular, we have:
$$
\CB_{(0,\dots,0)} = U_q(\widehat{\mathfrak{gl}}_n ) 
$$
Thus we encounter a particular instance of Problem \ref{prob:1}: when $Q$ is the cyclic quiver of length $n$, the statement of the Problem holds with $Q_{(m_1,\dots,m_n)}$ being a disjoint union of cyclic quivers of lengths $n_1,\dots,n_d$. \\

\subsection{} With Problem \ref{prob:1} in mind, we will now focus on the algebra $\CB_{\b0}$. \\

\begin{definition} 
\label{def:kac}

The Kac polynomial of $Q$ in dimension $\bn \in \nn$, denoted by:
\begin{equation}
\label{eqn:kac}
A_{Q,\bn}(t)
\end{equation}
counts the number of isomorphism classes of $\bn$-dimensional absolutely indecomposable representations of the quiver $Q$ over a finite field with $t$ elements. \\

\end{definition}

\noindent It was shown in \cite{K} that the number of absolutely indecomposable representations above is a polynomial in $t$, and so \eqref{eqn:kac} lies in $\BZ[t]$. This was further shown in \cite{HLRV} to lie in $\BN[t]$, thus opening the door to the notion that $A_{Q,\bn}(1)$ counts ``something". Before we make a conjecture as to what this something is, let us assemble all the Kac polynomials into a power series:
$$
A_Q(t,\bz) = \sum_{\bn \in \nn \backslash \b0} A_{Q,\bn}(t) \bz^{\bn} 
$$
where $\bz^{\bn} = \prod_{i \in I} z_i^{n_i}$. Similarly, define:
\begin{equation}
\label{eqn:graded dimension}
\chi_{\CB_{\b0}}(\bz) = \sum_{\bn \in \nn} \dim \CB_{\b0|\bn} \bz^{\bn} 
\end{equation}
and consider the plethystic exponential:
$$
\text{Exp} \left[ \sum_{\bn \in \nn \backslash \b0} d_{\bn} \bz^{\bn} \right] := \prod_{\bn \in \nn \backslash \b0} \frac 1{(1-\bz^{\bn})^{d_{\bn}}} 
$$
for any collection of natural numbers $\{d_{\bn}\}_{\bn \in \nn \backslash \b0}$. \\

\begin{conjecture}
\label{conj:kac poly}

For any quiver $Q$, we have:
\begin{equation}
\label{eqn:kac poly}
\chi_{\CB_{\b0}}(\bz) = \emph{Exp} \left[ A_Q(1,\bz) \right]
\end{equation}
In other words, $\CB_{\b0}$ is isomorphic (as a graded vector space) to the symmetric algebra of a graded vector space of graded dimension $A_Q(1,\bz)$. \\

\end{conjecture}

\noindent Conjecture \ref{conj:kac poly} gives an elementary combinatorial formula for the Kac polynomial at $t=1$, since $\dim \CB_{\b0|\bn}$ is the dimension of the vector space of Laurent polynomials satisfying the wheel conditions \eqref{eqn:wheel} and the growth conditions \eqref{eqn:slope plus} for $\bm = \b0$. We computed these dimensions using mathematical software and verified Conjecture \ref{conj:kac poly} in the cases when: \\

\begin{itemize}[leftmargin=*]

\item $Q$ is the quiver with one vertex and $g \in \{1,2,3\}$ loops, up to dimension $n=5$; \\

\item $Q$ is the quiver with two vertices and $d \in \{1,2,3,4\}$ edges between them, up to dimension vector $(n_1,n_2) = (3,3)$. \\

\end{itemize}

\noindent The two types of quivers above are relevant because they control the various wheel conditions \eqref{eqn:wheel}. Moreover, in the particular cases $g = 1$, $d = 1$ and $d = 2$ of the two bullets above, we have:
$$
\CB_{\b0} = U_{q}(\widehat{\fgl}_1), \qquad \CB_{\b0} = U_{q}(\fsl_3) \qquad \text{and} \qquad \CB_{\b0} = U_{q}(\widehat{\fsl}_2)
$$
respectively. In these cases, Conjecture \ref{conj:kac poly} is easily verified. \\

\begin{remark}

We may consider the subspace $\CB_{\b0}^{\emph{prim}} \subset \CB_{\b0}$ of primitive elements, i.e. those for which the coproduct $\Delta_{\b0}$ has no intermediate terms:
\begin{equation}
\label{eqn:primitive}
\Delta_{\b0}(P) = P \otimes 1 + h_{\emph{hdeg} P} \otimes P
\end{equation}
We expect $\dim \CB_{\b0|\bn}^{\emph{prim}}$ to be given by the number $C_{Q,\bn}(1)$ of \cite{BoS}, for all $\bn \in \nn$. \\

\noindent In particular, if $Q$ is the quiver with one vertex and $g\geq 2$ loops (such a vertex is called hyperbolic in the language of \loccitt), then the $q \rightarrow 1$ limit of $\CB_{\b0}$ should be the free Lie algebra on the vector space which is the $q \rightarrow 1$ limit of $\CB_{\b0}^{\emph{prim}}$. We thank Andrei Okounkov and Olivier Schiffmann for pointing out this expectation. \\

\end{remark}

\subsection{} We will now present two more frameworks which are conjecturally related to our constructions. The first of these is the $K$-theoretic Hall algebra of the quiver $Q$. To define it, let us consider the stack of $\bn$-dimensional representations of $Q$, for any $\bn \in \nn$:
\begin{equation}
\label{eqn:stack}
\fZ_{\bn} = \bigoplus_{\oij = e \in E} \text{Hom}(V_i, V_j) \Big / \prod_{i \in I} GL(V_i)
\end{equation}
where $V_i$ denotes a vector space dimension $n_i$, for every $i \in I$. The action of the product of general linear groups in \eqref{eqn:stack} is by conjugating homomorphisms $V_i \rightarrow V_j$. \\

\noindent The Kac polynomial of Definition \ref{def:kac} counts the number of (certain) points of $\fZ_{\bn}$ over the field with $t$ elements. There are also other fruitful ways to count points of the stack \eqref{eqn:stack}, but one can obtain similarly beautiful constructions by looking at other enumerative invariants of $\fZ_{\bn}$. For example, Schiffmann-Vasserot consider the equivariant algebraic $K$-theory groups of the cotangent bundle of the stack $\fZ_{\bn}$:
\begin{equation}
\label{eqn:k-ha}
K = \bigoplus_{\bn \in \nn} K_T(T^*\fZ_{\bn})
\end{equation}
where the torus $T = \BC^* \times \prod_{e \in E} \BC^*$ acts on $T^*\fZ_{\bn}$ as follows: the first factor of $\BC^*$ scales the cotangent fibers, and the $e$-th $\BC^*$ in the product scales the homomorphism corresponding to the same-named edge $e$ in \eqref{eqn:stack}. As $K$ is a module over the ring:
\begin{equation}
\label{eqn:k-th point}
K_T(\text{point}) = \BZ[q^{\pm 1}, t_e^{\pm 1}]_{e\in E}
\end{equation}
we may consider its localization with respect to the fraction field $\BF = \BQ(q, t_e)_{e \in E}$:
\begin{equation}
\label{eqn:localization}
K_{\loc} = K \bigotimes_{\BZ[q^{\pm 1}, t_e^{\pm 1}]_{e\in E}}  \BQ(q, t_e)_{e \in E}
\end{equation}
We refer to \cite{S Lect} for a survey of $K$-theoretic Hall algebras, to \cite[Section 2]{N} for a quick overview in notation similar to ours, and to Remark \ref{rem:cotangent stack} for an explicit presentation of $T^*\fZ_{\bn}$. In particular, the reason for summing over all $\bn$ in \eqref{eqn:k-ha} is to make $K$ into a $\nn$-graded $\BF$-algebra. Moreover, we have a $\BF$-algebra homomorphism:
$$
K_{\loc} \xrightarrow{\iota} \CV
$$
where $\CV$ is the algebra \eqref{eqn:big shuffle}. It was shown in \cite{VV} that $\iota$ is injective, in \cite{Zhao} that $\text{Im }\iota \subseteq \CS$, and in \cite{N} that $\text{Im }\iota = \CS$. We therefore have a $\BF$-algebra isomorphism:
\begin{equation}
\label{eqn:hom to shuffle}
K_{\loc} \xrightarrow{\sim} \CS
\end{equation}

\begin{problem}
\label{prob:2}

What is the geometric meaning of the slope subalgebras $\CB^+_{\bm} \subset \CA^+ = \CS$, for various $\bm \in \qq$, when pulled back to $K_{\emph{loc}}$ via the isomorphism \eqref{eqn:hom to shuffle}? \\

\end{problem}

\subsection{} A bridge between Conjecture \ref{conj:kac poly} and Problem \ref{prob:2} is provided by the work of Davison and Meinhardt (\cite{Dav, DM}), who studied the version of \eqref{eqn:k-ha} when equivariant $K$-theory is replaced by Borel-Moore homology. The resulting object $H$ is called the cohomological Hall algebra of the quiver $Q$, and its study goes back to Kontsevich and Soibelman in \cite{KS}. The algebra $H$ is related to the algebra $K$ as Yangians are related to quantum loop groups. For a general quiver $Q$, Davison and Meinhardt construct in \cite{Dav, DM} a $\nn$-graded Lie algebra $\fg_{\text{BPS}}$ with an algebra embedding:
$$
U(\fg_{\text{BPS}}) \subset H
$$
The Lie algebra $\fg_{\text{BPS}}$ has the following graded dimension (see \eqref{eqn:graded dimension}):
$$
\chi_{\fg_{\text{BPS}}}(\bz) = A_Q(t,\bz)  \quad \Rightarrow \quad  \chi_{U(\fg_{\text{BPS}})}(\bz) = \text{Exp} \left[ A_Q(t,\bz) \right]
$$
where $t$ keeps track of the homological degree on $H$. Therefore, it is natural to conjecture that the degeneration map $K \leadsto H$ sends $\CB_{\b0} \leadsto U(\fg_{\text{BPS}})$. The homological grading, ubiquitous on the $H$-side, is not readily seen on the $K$-side. This is why we expect \eqref{eqn:kac poly} to only see the value at $t = 1$ of the Kac polynomial $A_Q(t,\bz) \in \BN[t][[\bz]]$. \\

\subsection{}
\label{sub:nakajima}

We will now recall the construction of Nakajima quiver varieties associated to the quiver $Q$ (\cite{Nak 0}). To define these, consider for any $\bv, \bw \in \nn$ the affine space:
$$
N_{\bv,\bw} = \bigoplus_{\oij = e \in E} \Big[ \Hom(V_i, V_j) \oplus  \Hom(V_j, V_i) \Big] \bigoplus_{i \in I} \Big[ \Hom(W_i, V_i) \oplus \Hom(V_i, W_i) \Big]
$$
where $V_i$ (respectively $W_i$) are vector spaces of dimension $v_i$ (respectively $w_i$) for all $i \in I$. Points of the affine space above will be denoted by quadruples:
\begin{equation}
\label{eqn:quadruples}
(X_e,Y_e,A_i,B_i)_{e \in E, i\in I}
\end{equation}
where $X_e, Y_e, A_i, B_i$ denote homomorphisms in the four types of Hom spaces that enter the definition of $N_{\bv,\bw}$. Consider the action of:
\begin{equation}
\label{eqn:gauge group}
G_{\bv} = \prod_{i \in I} GL(V_i)
\end{equation}
on $N_{\bv,\bw}$ by conjugating $X_e, Y_e$, left-multiplying $A_i$ and right-multiplying $B_i$. It is easy to see that $G_{\bv}$ acts freely on the open locus of \textbf{stable} \footnote{In general, stability conditions for quiver varieties are indexed by $\bth \in \mathbb{R}^I$; the one studied herein corresponds to $\bth = (1,\dots,1)$} points:
$$
N_{\bv,\bw}^s \subset N_{\bv,\bw}
$$
i.e. those \eqref{eqn:quadruples} such that there does not exist a collection of subspaces $\{V_i' \subseteq V_i\}_{i \in I}$ (other than $V_i' = V_i$ for all $i \in I$) which is preserved by the maps $X_e$ and $Y_e$, and contains $\text{Im }A_i$ for all $i \in I$. Let us consider the quadratic moment map:
\begin{equation}
\label{eqn:moment}
N_{\bv,\bw} \xrightarrow{\mu} \text{Lie } G_{\bv} = \bigoplus_{i \in I} \text{Hom}(V_i,V_i)
\end{equation}
$$
\mu ( (X_e,Y_e,A_i,B_i)_{e\in E, i\in I}) = \sum_{e \in E} \Big(X_e Y_e - Y_e X_e \Big) + \sum_{i \in I} A_iB_i
$$
If we write $\mu^{-1}_{\bv,\bw}(0)^s = \mu^{-1}_{\bv,\bw}(0) \cap N_{\bv,\bw}^s$, then there is a geometric quotient:
\begin{equation}
\label{eqn:quiver var}
\CN_{\bv,\bw} = \mu^{-1}_{\bv,\bw}(0)^s / GL_{\bv}
\end{equation}
which is called the \textbf{Nakajima quiver variety} for the quiver $Q$, associated to $\bv,\bw$. \\

\begin{remark}
\label{rem:cotangent stack}

The $\bv = \bn$, $\bw = \b0$ version of the construction above, where instead of taking the geometric quotient \eqref{eqn:quiver var} one takes the stack quotient, is simply $T^*\fZ_{\bn}$. \\

\end{remark}

\subsection{} 
\label{sub:k-theory}

The algebraic group:
\begin{equation}
\label{eqn:big torus}
T_{\bw} = \BC^* \times \prod_{e \in E} \BC^* \times \prod_{i \in I} GL(W_i)
\end{equation}
acts on Nakajima quiver varieties as follows:
$$
\left( \bar{q}, \bar{t}_e, \bar{U}_{i} \right)_{e\in E, i\in I} \cdot (X_e,Y_e,A_i,B_i)_{e \in E, i \in I} = \left(\frac {X_e}{\bar{t}_e}, \frac {\bar{t}_e Y_e}{\bar{q}}, A_i \bar{U}_{i}^{-1}, \frac {\bar{U}_{i} B_i}{\bar{q}} \right)_{e\in E, i \in I}
$$
With respect to the action above, the $T_{\bw}$-equivariant algebraic $K$-theory groups of Nakajima quiver varieties are modules over the ring: 
$$
K_{T_{\bw}}(\text{point}) =  \BZ[q^{\pm 1}, t_e^{\pm 1}, u_{ia}^{\pm 1}]^{\text{sym}}_{e\in E,i\in I, 1 \leq a \leq w_i}
$$
(where ``sym" means symmetric in the equivariant parameters $u_{i1},\dots,u_{iw_i}$ for each $i \in I$ separately). We will localize our $K$-theory groups by analogy with \eqref{eqn:localization}:
\begin{equation}
\label{eqn:k v w}
K_{\bv, \bw} = K_{T_{\bw}}(\CN_{\bv,\bw}) \bigotimes_{\BZ[q^{\pm 1}, t_e^{\pm 1}, u_{ia}^{\pm 1}]^{\sym}_{e\in E,i\in I, 1 \leq a \leq w_i}}  \BQ(q, t_e, u_{ia})^{\sym}_{e \in E,i\in I, 1\leq a\leq w_i}
\end{equation}
As with the $K$-theoretic Hall algebra, it makes sense to consider the direct sum:
\begin{equation}
\label{eqn:k w}
K_{\bw} = \bigoplus_{\bv \in \nn} K_{\bv,\bw}
\end{equation}
For every $i \in I$, consider the \textbf{tautological bundle} $V_i$ of rank $v_i$, whose fiber over a point \eqref{eqn:quadruples} is the vector space $V_i$ itself; this is a non-trivial vector bundle, because Nakajima quiver varieties arise as quotients by the group \eqref{eqn:gauge group}. We formally write:
$$
[V_i] = x_{i1}+\dots+x_{iv_i} \in K_{\bv,\bw}
$$
The symbols $x_{ia}$ are not elements of $K_{\bv,\bw}$, but any symmetric Laurent polynomial in them is (specifically, it is obtained by taking the $K$-theory class of an appropriate Schur functor of the tautological vector bundle $V_i$). We will abbreviate:
$$
\bX_{\bv} = \{\dots, x_{i1},\dots, x_{iv_i},\dots\}_{i \in I}
$$
By the discussion above, any Laurent polynomial $p(\dots,x_{i1},\dots,x_{iv_i},\dots)$ which is symmetric in the $x_{ia}$'s (for each $i \in I$ separately) yields an element of $K$-theory:
\begin{equation}
\label{eqn:tautological class}
p(\bX_{\bv}) \in  K_{\bv,\bw}
\end{equation}
called a \textbf{tautological class}. By \cite[Theorem 1.2]{MN}, tautological classes (as $p$ runs over all symmetric Laurent polynomials) linearly span $K_{\bv,\bw}$ for any $\bv, \bw \in \nn$. \\






\begin{example}

Recall the function $\zeta_{ij}(x)$ of \eqref{eqn:def zeta}, and consider its close cousin:
\begin{equation}
\label{eqn:def tzeta}
\tzeta_{ij}(x) = \frac {\zeta_{ij}(x)}{\left(1-\frac xq\right)^{\delta_j^i}\left(1-\frac 1{qx}\right)^{\delta_j^i}} = \frac {\prod_{e = \oij \in E} \left(\frac 1{t_e} - x \right) \prod_{e = \oji \in E} \left(1 - \frac {t_e}{qx} \right)}{\left(1-x \right)^{\delta_j^i} \left(1 - \frac 1{qx} \right)^{\delta_j^i}}
\end{equation}
Then for any $\bn \in \nn$, let us define:
\begin{align}
&\tzeta\left(\frac {\bZ_{\bn}}{\bX_{\bv}}\right) = \prod^{i \in I}_{1 \leq a \leq n_i} \prod^{j \in I}_{1 \leq b \leq v_j} \frac {\prod_{e = \oij \in E} \left(\frac 1{t_e} - \frac {z_{ia}}{x_{jb}} \right) \prod_{e = \oji \in E} \left(1 - \frac {t_e x_{jb}}{qz_{ia}} \right)}{\left(1-\frac {z_{ia}}{x_{jb}} \right)^{\delta_j^i} \left(1 - \frac {x_{jb}}{qz_{ia}} \right)^{\delta_j^i}} \label{eqn:zeta plus} \\
&\tzeta\left(\frac {\bX_{\bv}}{\bZ_{\bn}}\right) = \prod^{i \in I}_{1 \leq a \leq n_i} \prod^{j \in I}_{1 \leq b \leq v_j} \frac {\prod_{e = \oji \in E} \left(\frac 1{t_e} - \frac {x_{jb}}{z_{ia}} \right) \prod_{e = \oij \in E} \left(1 - \frac {t_e z_{ia}}{q x_{jb}} \right)}{\left(1-\frac {x_{jb}}{z_{ia}} \right)^{\delta_j^i} \left(1 - \frac {z_{ia}}{qx_{jb}} \right)^{\delta_j^i}} \label{eqn:zeta minus}
\end{align}
as elements of $K_{\bv,\bw}[[\dots, z^{\pm 1}_{i1},\dots,z^{\pm 1}_{in_i},\dots]]$. \\



\end{example}


\subsection{}
\label{sub:nakajima correspondencs}

One important reason for considering all $\bv$'s together in \eqref{eqn:k w} is given by the following correspondences of Nakajima, which give operators $K_{\bv^+,\bw} \leftrightharpoons K_{\bv^-,\bw}$ whenever $\bv^+ = \bv^- + \bs^i$ for some $i \in I$. Explicitly, for any such $\bv^\pm$, let:
\begin{equation}
\label{eqn:commutative diagram}
\xymatrix{& \CN_{\bv^+,\bv^-,\bw} \ar[ld]_{\pi_+} \ar[rd]^{\pi_-} & \\
\CN_{\bv^+,\bw} & & \CN_{\bv^-,\bw}}
\end{equation}
be the Hecke correspondences of \cite[Section 5]{Nak}. There is a tautological line bundle:
$$
\CL_i \in \text{Pic} \left( \CN_{\bv^+,\bv^-,\bw} \right) \qquad \leadsto \qquad l_i = [\CL_i] \in K_{T_{\bw}}(\CN_{\bv^+,\bv^-,\bw} )
$$
With this notation in mind, let us consider the following endomorphisms of $K_{\bw}$:
$$
\begin{tikzcd}
K_{\bv^-,\bw} \arrow[rr, bend left]{}{E_{i,d}} & & K_{\bv^+,\bw} \arrow[ll, bend left]{}{F_{i,d}}
\end{tikzcd} \qquad \text{and} \begin{tikzcd}
 K_{\bv,\bw} \arrow[loop]{}[swap]{H_i^{\pm}(z)}
\end{tikzcd}
$$
(for all $\bv^+ = \bv^- + \bs^i$ and $\bv$ in $\nn$) given by the formulas:
\begin{align}
&E_{i,d} (\alpha) = \pi_{+*} \left(l_i^d \cdot \frac {\prod^{j \in I}_{e = \oij} t_e^{-v^+_j} \prod^{j \in I}_{e = \oji} \left(\det V^+_j \right) \left( \frac {-t_e}{l_i q} \right)^{v^+_j}}{\left(\det V_i^+\right) \left( \frac {-1}{l_i q}\right)^{v^+_i}}  \cdot \pi_-^*(\alpha) \right)  \label{eqn:op e} \\
&F_{i,d} (\alpha) = \pi_{-*} \left( l_i^d \cdot \frac {\prod^{j \in I}_{e = \oij} (\det V_j^-) \left( \frac {-q}{l_i t_e} \right)^{v^-_j} \prod^{j \in I}_{e = \oji} t_e^{-v^-_j}}{(\det W_i)^{-1} \left(-l_i \right)^{r_i} (\det V_i^-) \left(\frac {-q}{l_i} \right)^{v_i^-}} \cdot \pi_+^*(\alpha) \right)\label{eqn:op f} \\
&H_i^\pm(z_{i1}) (\alpha) = \frac {\tzeta\left(\frac {\bZ_{\bs^i}}{\bX_{\bv}}\right)}{\tzeta\left(\frac {\bX_{\bv}}{\bZ_{\bs^i}} \right)} \cdot \frac {\wedge^\bullet \left(\frac {z_{i1}q}{W_i} \right)}{\wedge^\bullet \left(\frac {z_{i1}}{W_i} \right)}  \cdot \alpha \label{eqn:op h} 
\end{align}
\footnote{The notation in \eqref{eqn:op h} is such that for any variable $z$ and any vector space $S$, we set:
$$
\wedge^\bullet\left(\frac zS \right) = \sum_{k = 0}^{\text{dim }S} (-z)^k [\wedge^k(S^\vee)] \qquad \text{and} \qquad \wedge^\bullet\left(\frac Sz \right) = \sum_{k = 0}^{\text{dim }S} (-z)^{-k} [\wedge^k(S)] 
$$} for all $i \in I$ and $d \in \BZ$. In formulas \eqref{eqn:op e}--\eqref{eqn:op f}, the fractions are the $K$-theory classes of certain line bundles on $\CN_{\bv^+,\bv^-,\bw}$ (built out of the determinants of the vector bundles $\{V_j^\pm,W_j\}_{j \in I}$, as well as the tautological line bundle) times equivariant constants. It was shown in \cite{Nak} that the operators \eqref{eqn:op e}--\eqref{eqn:op h} induce an action of the quantum loop group associated to the quiver $Q$ on $K_{\bw}$, in the case when the quiver has no edge loops (see also \cite{GV,VV 0} for earlier work on the cyclic quiver case). \\

\subsection{}

It is natural to expect the operators \eqref{eqn:op e} to extend to an action $K_{\loc} \curvearrowright K_{\bw}$, $\forall \bw \in \nn$ (more specifically, the operators $E_{i,d}$ should correspond to the action of the $\bn = \bs^i$ summand of \eqref{eqn:k-ha}). This was proved in complete generality in \cite{YZ}. As \cite{N} showed that the localized $K$-theoretic Hall algebra is isomorphic to the shuffle algebra $\CS$, it becomes natural to ask for a ``shuffle" version of formula \eqref{eqn:op e}, and analogously for \eqref{eqn:op f}. As shown in \cite{K-th}, the shuffle algebra naturally arises when we present formulas in terms of tautological classes \eqref{eqn:tautological class}. To this end, we have the following generalization of \cite[Theorem 4.7]{K-th} (which treated the Jordan quiver case). \\

\begin{theorem} 
\label{thm:action}

The following formulas give actions $\CA^{\pm} \curvearrowright K_{\bw}$:
\begin{align}
&p(\bX_{\bv}) \xrightarrow{F} \frac 1{\bn!} \int^+ \frac {F(\bZ_{\bn})}{\tzeta \left( \frac {\bZ_{\bn}}{\bZ_{\bn}} \right)} p(\bX_{\bv+\bn} - \bZ_{\bn}) \tzeta\left(\frac {\bZ_{\bn}}{\bX_{\bv+\bn}}\right) \wedge^\bullet \left(\frac {\bZ_{\bn} q}{\bW} \right)  \label{eqn:action plus} \\
&p(\bX_{\bv}) \xrightarrow{G} \frac 1{\bn!} \int^- \frac {G(\bZ_{\bn})}{\tzeta \left( \frac {\bZ_{\bn}}{\bZ_{\bn}} \right)} p(\bX_{\bv - \bn} + \bZ_{\bn}) \tzeta\left(\frac {\bX_{\bv - \bn}}{\bZ_{\bn}} \right)^{-1} \wedge^\bullet \left(\frac {\bZ_{\bn}}{\bW} \right)^{-1}  \label{eqn:action minus}
\end{align}
for all $F \in \CA_{\bn}$, $G \in \CA_{- \bn}$ (the notation will be explained after the statement of the Theorem). Together with \eqref{eqn:op h}, formulas \eqref{eqn:action plus}--\eqref{eqn:action minus} glue to an action $\CA \curvearrowright K_{\bw}$. \\

\end{theorem}

\noindent Let us now explain the notation in \eqref{eqn:action plus}--\eqref{eqn:action minus}, except for the definition of the integrals $\int^\pm$, which will be given in Subsection \ref{sub:normal}. We write:
$$
(F \text{ or } G)(\bZ_{\bn}) = (F \text{ or } G)(\dots,z_{i1},\dots,z_{in_i},\dots)_{i \in I} 
$$
$$
\tzeta \left( \frac {\bZ_{\bn}}{\bZ_{\bn}} \right) = \mathop{\prod^{i, j \in I}_{1 \leq a \leq n_i, 1 \leq b \leq n_j}}_{(i,a) \neq (j,b)} \frac {\zeta_{ij} \left( \frac {z_{ia}}{z_{jb}} \right)}{\left( 1 - \frac {z_{ia}}{z_{jb}q} \right)^{\delta_j^i}} 
$$
and:
$$\wedge^\bullet \left(\frac {\bZ_{\bn}(1 \text{ or } q)}{\bW} \right) = \prod^{i\in I}_{1\leq a \leq n_i} \wedge^\bullet \left(\frac {z_{ia}(1 \text{ or } q)}{W_i} \right) 
$$
The operation:
\begin{equation}
\label{eqn:plethysm plus}
p(\bX_{\bv}) \mapsto p(\bX_{\bv - \bn} + \bZ_{\bn})
\end{equation}
is called a \textbf{plethysm}, and it is defined by evaluating the Laurent polynomial $p$ at the collection of variables $\{\dots,x_{i1},\dots,x_{i,v_i-n_i},z_{i1},\dots,z_{in_i},\dots\}_{i \in I}$. The notation:
\begin{equation}
\label{eqn:plethysm minus}
p(\bX_{\bv}) \mapsto p(\bX_{\bv+\bn} - \bZ_{\bn})
\end{equation}
would like to refer to the ``inverse" operation of \eqref{eqn:plethysm plus}, but the problem is that it is not uniquely defined. Indeed, by the fundamental theorem of symmetric polynomials, the Laurent polynomial $p$ can be written as:
\begin{equation}
\label{eqn:plethysm minus 1}
p(\bX_{\bv}) = \frac {\text{polynomial in } \{x_{i1}^s + \dots + x_{iv_i}^s\}_{i \in I, s\in \BN^*}}{\prod_{i \in I} \left(x_{i1} \dots x_{iv_i} \right)^N}
\end{equation}
in infinitely many ways, for various polynomials in the numerator and various natural numbers $N$ in the denominator. We define \eqref{eqn:plethysm minus} by:
\begin{equation}
\label{eqn:plethysm minus 2}
p(\bX_{\bv+\bn} - \bZ_{\bn}) = \frac {\text{polynomial in } \{x_{i1}^s + \dots + x_{i,v_i+n_i}^s - z_{i1}^s - \dots - z_{in_i}^s\}_{i \in I, s\in \BN^*}}{\prod_{i \in I} \left(\frac {x_{i1} \dots x_{i,v_i+n_i}}{z_{i1}\dots z_{in_i}} \right)^N}
\end{equation}
Of course, the right-hand side of the expressions above depends on the particular polynomial and the number $N$ in \eqref{eqn:plethysm minus 1}, but we will show in the proof of Theorem \ref{thm:action} that the right-hand side of \eqref{eqn:action plus} does not depend on these choices. \\

\subsection{}
\label{sub:normal}

\noindent The integrals \eqref{eqn:action plus}--\eqref{eqn:action minus} are defined by (see \cite[Definition 3.15]{Laumon}):
\begin{align}
&\int^+ T(\dots, z_{ia},\dots ) = \sum^{\text{functions}}_{\sigma : \{(i,a)\} \rightarrow \{\pm 1\}} \int_{|z_{ia}| = r^{\sigma(i,a)}}^{|q/t_e|^{\pm 1}, |t_e|^{\pm 1} > 1} T(\dots, z_{ia},\dots ) \prod_{(i,a)} \frac { \sigma(i,a)dz_{ia}}{2\pi \sqrt{-1} z_{ia}} \label{eqn:normal int plus} \\
&\int^- T(\dots, z_{ia},\dots ) = \sum^{\text{functions}}_{\sigma : \{(i,a)\} \rightarrow \{\pm 1\}} \int_{|z_{ia}| = r^{\sigma(i,a)}}^{|q/t_e|^{\pm 1}, |t_e|^{\pm 1} < 1} T(\dots, z_{ia},\dots ) \prod_{(i,a)} \frac { \sigma(i,a)dz_{ia}}{2\pi \sqrt{-1} z_{ia}} \label{eqn:normal int minus}
\end{align}
for some positive real number $r \ll 1$. In each summand, each variable $z_{ia}$ is integrated over either a very small circle of radius $r$ or a very large circle of radius $r^{-1}$. The meaning of the superscripts ``$|q/t_e|^{\pm 1}, |t_e|^{\pm 1} > 1$" that adorn the integral \eqref{eqn:normal int plus} is the following: in the summand corresponding to a particular $\sigma$, if:
$$
\sigma(i,a) = \sigma(j,b) = 1 \qquad (\text{respectively } \sigma(i,a) = \sigma(j,b) = -1)
$$
then the variables $z_{ia}$ and $z_{jb}$ are both integrated over the small (respectively large) circle. The corresponding integral is computed via residues under the assumption $|q/t_e|, |t_e| > 1$ (respectively $|q/t_e|^{-1}, |t_e|^{-1} > 1$). If $\sigma(i,a) \neq \sigma(j,b)$, then we do not need to assume anything about the sizes of $q$ and $t_e$. One defines \eqref{eqn:normal int minus} analogously. \\

\noindent The following Proposition is precisely the motivation behind our definition of $\int^\pm$, and its proof closely follows the analogous computation in \cite[Theorem 3.17]{Laumon}. \\

\begin{proposition}
\label{prop:normal}

The right-hand side of formula \eqref{eqn:action plus}, respectively \eqref{eqn:action minus}, for:
\begin{equation}
\label{eqn:special}
F = e_{i_1,d_1} * \dots * e_{i_n,d_n} \qquad \text{respectively} \qquad G = f_{i_1,d_1} * \dots * f_{i_n,d_n}
\end{equation}
is equal to the composition of the right-hand sides of \eqref{eqn:action plus} for $F = e_{i_1,d_1}$, \dots , $F = e_{i_n,d_n}$, respectively the right-hand sides of \eqref{eqn:action minus} for $G = f_{i_1,d_1}$, \dots , $G = f_{i_n,d_n}$. \\

\end{proposition}

\noindent Indeed, it is easy to see that the composition of the right-hand sides of formulas \eqref{eqn:action plus} and \eqref{eqn:action minus} for $F = e_{i_1,d_1}$, \dots, $F = e_{i_n,d_n}$ and $G = f_{i_1,d_1}$, \dots, $G = f_{i_n,d_n}$ is:
\begin{multline}
\int_{\{0,\infty\} \succ z_1 \succ \dots \succ z_n} \frac {z_1^{d_1} \dots z_n^{d_n}}{\prod_{1 \leq a < b \leq n} \tzeta_{i_bi_a} \left( \frac {z_b}{z_a} \right)} \\ p(\bX_{\bv+\bn} - \bZ_{\bn}) \tzeta\left(\frac {\bZ_{\bn}}{\bX_{\bv+\bn}}\right) \wedge^\bullet \left(\frac {\bZ_{\bn} q}{\bW} \right) \prod_{a = 1}^n \frac {dz_a}{2\pi \sqrt{-1}z_a}   \label{eqn:action plus explicit}
\end{multline}
\begin{multline}
\int_{\{0,\infty\} \succ z_1 \succ \dots \succ z_n} \frac {z_{1}^{d_1} \dots z_{n}^{d_n}}{\prod_{1 \leq a < b \leq n} \tzeta_{i_ai_b} \left( \frac {z_a}{z_b} \right)} \\ p(\bX_{\bv-\bn} + \bZ_{\bn}) \tzeta\left(\frac {\bX_{\bv-\bn}}{\bZ_{\bn}} \right)^{-1} \wedge^\bullet \left(\frac {\bZ_{\bn}}{\bW} \right)^{-1} \prod_{a = 1}^n \frac {dz_a}{2\pi \sqrt{-1}z_a} \label{eqn:action minus explicit}
\end{multline}
The notation $\int_{\{0,\infty\} \succ z_1 \succ \dots \succ z_n}$ means that the variable $z_1$ is integrated over a contour in the complex plane which surrounds $0$ and $\infty$, the variable $z_2$ is integrated over a contour which surrounds the previous contour etc., and the contours are also far away from each other compared to the size of the equivariant parameters $q,t_e$. \\

\noindent Moreover, in formulas \eqref{eqn:action plus explicit}--\eqref{eqn:action minus explicit}, we implicitly identify the variables:
$$
\{z_1,\dots,z_n\} \leftrightarrow \{...,z_{i1},\dots,z_{in_i},\dots\}_{i \in I}
$$
by mapping $z_a$ in a one-to-one way to some $z_{i_a \bullet}$ (the specific choice of $\bullet \in \BN$ does not matter due to the symmetry of all expressions involved in the variables which make up $\bZ_{\bn}$). Note that we need $n = |\bn|$ in order for this notation to be consistent. We leave the equivalence of \eqref{eqn:action plus}--\eqref{eqn:action minus} for the shuffle elements \eqref{eqn:special} with formulas \eqref{eqn:action plus explicit}--\eqref{eqn:action minus explicit} as an exercise to the interested reader (it closely follows the analogous computation in \cite[Proof of Theorem 3.17]{Laumon}, which dealt with a close relative of our construction in the particular case when $Q$ is the cyclic quiver). 

\begin{proof} \emph{of Theorem \ref{thm:action}:} Our main task will be to establish the following claim: \\

\begin{claim}
\label{claim:proof}

The case $F = e_{i,d}$ of \eqref{eqn:action plus} yields the same formula as \eqref{eqn:op e}, $\forall i \in I$, $d \in \BZ$. Similarly, the case $G = f_{i,d}$ of \eqref{eqn:action minus} yields the same formula as \eqref{eqn:op f}. \\

\end{claim}

\noindent This claim establishes the fact that formulas \eqref{eqn:action plus}--\eqref{eqn:action minus} yield well-defined operators on $K_{\bw}$ when $F = e_{i,d}$ and $G = f_{i,d}$, respectively. The meaning of the phrase ``well-defined" in the previous sentence is that:
$$
\text{if } p(\bX_{\bv}) = 0, \text{ then the RHS of \eqref{eqn:action plus}--\eqref{eqn:action minus} is also }0
$$
(in particular, the right-hand side of \eqref{eqn:action plus} does not depend on the choices we made in defining \eqref{eqn:plethysm minus}).  By Proposition \ref{prop:normal}, formulas \eqref{eqn:action plus} and \eqref{eqn:action minus} also yield well-defined operators on $K_{\bw}$ for any $F$ and $G$ of the form \eqref{eqn:special}. Since by Theorem \ref{thm:generate}, any element of $\CS$ and $\CS^{\op}$ is a linear combination of such $F$'s and $G$'s, this proves that formulas \eqref{eqn:action plus} and \eqref{eqn:action minus} are well-defined for any $F \in \CA^+$ and any $G \in \CA^-$. Similarly, the fact that the aforementioned formulas are multiplicative in $F$ and $G$ (thus implying that \eqref{eqn:action plus}--\eqref{eqn:action minus} yield actions $\CA^\pm \curvearrowright K_{\bw}$) is an immediate consequence of Proposition \ref{prop:normal}. To prove that the actions $\CA^\pm \curvearrowright K_{\bw}$ glue to an action $\CA \curvearrowright K_{\bw}$, one only needs to check relations \eqref{eqn:rel double 1}--\eqref{eqn:rel double 3}; this closely follows the $Q = $ cyclic quiver case treated in \cite[Theorem II.9.]{Thesis}. \\

\begin{proof} \emph{of Claim \ref{claim:proof}:} Consider the following complex of vector bundles on $\CN_{\bv,\bw}$:
\begin{equation}
\label{eqn:complex}
U_i = \left[V_i \cdot q \xrightarrow{(B_i,-X_e,Y_{e'})} W_i \oplus \bigoplus_{e = \oij} V_j \cdot \frac q{t_e} \oplus \bigoplus_{e' = \oji} V_j \cdot t_e \xrightarrow{(A_i,Y_e,X_{e'})} V_i  \right]
\end{equation}
(which originated in \cite{Nak 0}) with the middle term in homological degree 0. By the stability condition, the second arrow is point-wise surjective, and thus its kernel $K_i$ is a vector bundle; thus $U_i$ is quasi-isomorphic to a complex $[V_i \cdot q \rightarrow K_i]$ of vector bundles. For such a complex, we may construct the projectivization: 
$$
\BP_{\CN_{\bv,\bw}}(U_i) = \text{Proj}_{\CN_{\bv,\bw}}(\text{Sym}^\bullet(U_i))
$$
as a dg-scheme over $\CN_{\bv,\bw}$ (see \cite[Section 5.18]{Shuf surf} for our notational conventions), and analogously for the dual complex $U_i^\vee[1] \cdot q = [K_i^\vee \cdot q \rightarrow V_i^\vee]$. With this in mind, it is well-known that we have isomorphisms:
\begin{align}
&\CN_{\bv^+,\bv^-,\bw} \cong \BP_{\CN_{\bv^-,\bw}} (U_i) \label{eqn:proj minus} \\
&\CN_{\bv^+,\bv^-,\bw} \cong \BP_{\CN_{\bv^+,\bw}} (U_i^\vee[1] \cdot q) \label{eqn:proj plus} 
\end{align}
with respect to which the line bundle $\CL_i$ is isomorphic to $\CO(1)$ and $\CO(-1)$, respectively. A straightforward computation, which follows directly from the well-known formulas in \cite[Proposition 5.19]{Shuf surf}, yields for any tautological class $p$ as in \eqref{eqn:tautological class}:

\begin{align*}
&E_{i,d}(p(\bX_{\bv^-})) = \int^+ z_{i1}^d  p(\bX_{\bv^+} - \bZ_{\bs^i}) \frac {\prod^{j \in I}_{e = \oij} t_e^{-v_j} \prod^{j \in I}_{e = \oji} \left(\det V_j \right) \left( \frac {-t_e}{z_{i1} q} \right)^{v_j}}{\left(\det V_i\right) \left( \frac {-1}{z_{i1} q}\right)^{v_i}}  \wedge^\bullet \left(\frac {z_{i1}q}{U_i} \right)  \\
&F_{i,d}(p(\bX_{\bv^+})) = \int^- z_{i1}^d p(\bX_{\bv^-} + \bZ_{\bs^i}) \frac {\prod^{j \in I}_{e = \oij} (\det V_j) \left( \frac {-q}{z_{i1} t_e} \right)^{v_j} \prod^{j \in I}_{e = \oji} t_e^{-v_j}}{(\det W_i)^{-1} \left(-z_{i1} \right)^{r_i} (\det V_i) \left(\frac {-q}{z_{i1}} \right)^{v_i}} \wedge^\bullet \left(- \frac {U_i}{z_{i1}} \right)
\end{align*}
One can express $U_i$ in terms of the vector bundles $V_i$ and the trivial vector bundles $W_i$ using \eqref{eqn:complex}, and one notices that the right-hand sides of the formulas above are precisely the right-hand sides of \eqref{eqn:action plus} and \eqref{eqn:action minus} when $F = e_{i,d}$ and $G = f_{i,d}$.


\end{proof}

\end{proof}

\subsection{}
\label{sub:coproducts}

As a consequence of Theorem \ref{thm:action}, we obtain an algebra homomorphism:
\begin{equation}
\label{eqn:action product}
\CA \xrightarrow{\Psi} \prod_{\bw \in \nn} \text{End}(K_{\bw})
\end{equation} 
To make the map $\Psi$ into a bialgebra homomorphism, one needs to place a coproduct on the right-hand side, which interweaves the modules $K_{\bw}$ as $\bw$ varies over $\BN$. Consider any $\bw^1,\bw^2 \in \nn$, let $\bw = \bw^1+\bw^2$, and take the one-parameter subgroup:
\begin{equation}
\label{eqn:torus}
\BC^* \ni t \xrightarrow{\tau} \prod_{i \in I} \text{diag} (\underbrace{1,\dots,1}_{w^1_i \text{ times}},\underbrace{t,\dots,t}_{w_i^2 \text{ times}} ) \in \prod_{i\in I} GL(W_i)
\end{equation}
The fixed locus of $\tau$ acting on $\CN_{\bv,\bw}$ is:
\begin{equation}
\label{eqn:fixed locus}
\CN_{\bv,\bw}^{\BC^*} \cong \bigsqcup_{\bv^1+\bv^2 = \bv} \CN_{\bv^1,\bw^1} \times \CN_{\bv^2,\bw^2} \stackrel{\iota}\hookrightarrow \CN_{\bv,\bw}
\end{equation}
consisting of quadruples \eqref{eqn:quadruples} which respect fixed direct sum decompositions $V_i = V_i^1 \oplus V_i^2$ and $W_i = W_i^1 \oplus W_i^2$. We have a decomposition of the normal bundle: 
$$
T_{\CN_{\bv,\bw}/\CN_{\bv^1,\bw^1} \times \CN_{\bv^2,\bw^2}} = T^+ \oplus T^-
$$
where $T^+$ (respectively $T^-$) consists of the attracting (respectively repelling) sub-bundles with respect to the action of the one-parameter subgroup $\tau$ of \eqref{eqn:torus}. Because $\tau$ preserves the holomorphic symplectic form on $\CN_{\bv,\bw}$ (which we have not defined), the sub-bundles $T^+$ and $T^-$ are dual to each other, and so have the same rank. This allows us to think of $T^+$ as ``half" of the normal bundle, and set:
\begin{equation}
\label{eqn:upsilon}
\Upsilon : K_{\bv,\bw} \xrightarrow{\wedge^\bullet \left(- T^{+ \vee} \right) \cdot \iota^*} \bigoplus_{\bv = \bv^1+\bv^2} K_{\bv^1,\bw^1} \otimes K_{\bv^2,\bw^2}
\end{equation}
 Conjugation with the (product over all $\bv,\bw \in \nn$ of the) map $\Upsilon$ yields a coproduct:
\begin{equation}
\label{eqn:end coproduct}
\prod_{\bw \in \nn} \text{End}(K_{\bw}) \longrightarrow \prod_{\bw^1 \in \nn} \text{End}(K_{\bw^1}) \ \widehat{\otimes} \ \prod_{\bw^2 \in \nn} \text{End}(K_{\bw^2})
\end{equation}
It is straightforward to check that the map $\Psi$ of \eqref{eqn:action product} intertwines the coproduct on $\CA$ of \eqref{eqn:cop 1}--\eqref{eqn:cop 3} with the coproduct \eqref{eqn:end coproduct}. Explicitly, this boils down to the commutativity of the following squares, which we leave as exercises to the interested reader:

$$
\xymatrixcolsep{5pc}\xymatrix{
K_{\bw} \ar[dd]_{\Upsilon}  \ar@/_2pc/[r]^{F_i(z)}  \ar@/^2pc/[r]^{E_i(z)} \ar[r]^{H^\pm_i(z)} & K_{\bw} \ar[dd]^{\Upsilon} \\ \\
K_{\bw^1} \otimes K_{\bw^2} \ar@/_2pc/[r]_{F_i(z) \otimes H^-_i(z) + 1 \otimes F_i(z)} \ar@/^2pc/[r]^{E_i(z) \otimes 1 + H^+_i(z) \otimes E_i(z)} \ar[r]^{H^\pm_i(z) \otimes H^\pm_i(z)} &K_{\bw^1} \otimes K_{\bw^2}}
$$
(above, let $E_i(z) = \sum_{d \in \BZ} \frac {E_{i,d}}{z^d}$ and $F_i(z) = \sum_{d \in \BZ} \frac {F_{i,d}}{z^d}$) for any $\bw = \bw^1+\bw^2$ in $\nn$. \\


\subsection{}

The construction of the previous Subsection (which is by now folklore among the experts) is quite straightforward, but unfortunately has a number of drawbacks. The first is that it heavily uses equivariant localization. The second is that it only produces the topological coproduct $\Delta$, instead of the more desirable coproducts discussed in Subsection \ref{sub:other coproducts} (chief among which is the Drinfeld-Jimbo coproduct). \\

\noindent To remedy these issues, \cite{Nak 2} suggested to consider the attracting subvariety of the fixed point locus $\CN_{\bv,\bw}^{\BC^*} \hookrightarrow \CN_{\bv,\bw}$, as a replacement for the class $\wedge^\bullet \left(- T^{+\vee} \right)$ in \eqref{eqn:upsilon}. Through a wide-reaching framework that pertains to conical symplectic resolutions, \cite{MO} defined a specific class on the disjoint union of all attracting subvarieties, called the \textbf{stable basis}, which gives a better analogue of the map \eqref{eqn:upsilon}. The $K$-theoretic version of this construction was developed in \cite{AO,O1,O2,OS}, thus yielding a map:
\begin{equation}
\label{eqn:stable basis}
\Upsilon_{\bm} : K_{\bv,\bw} \longrightarrow \bigoplus_{\bv = \bv^1+\bv^2} K_{\bv^1,\bw^1} \otimes K_{\bv^2,\bw^2}
\end{equation}
for any decomposition $\bw = \bw^1+\bw^2$ in $\nn$ and any $\bm \in \qq$. As explained in \cite{MO}, applying the FRT formalism to the maps \eqref{eqn:stable basis} gives rise to a Hopf algebra:
\begin{equation}
\label{eqn:mo}
U_q(\widehat{\mathfrak{g}}^Q) \subset \prod_{\bw \in \nn} \text{End}(K_{\bw})
\end{equation}
A well-known conjecture in the field (see \cite[Conjecture 1.2]{Pad 2} or \cite[Conjecture]{SV Coh} for various incarnations) posits that the integral version of the Hopf algebra \eqref{eqn:mo} is isomorphic to the double $K$-theoretic Hall algebra \eqref{eqn:k-ha}. As the localization \eqref{eqn:localization} was shown to be isomorphic to the shuffle algebra $\CS$ in \cite{N}, we propose the following. \\

\begin{conjecture}
\label{conj:mo}

The map \eqref{eqn:action product} yields an isomorphism $\CA \cong U_q(\widehat{\mathfrak{g}}^Q)$. \\

\end{conjecture}

\noindent We remark that the definition of $U_q(\widehat{\mathfrak{g}}^Q)$ relies on many choices that we do not recall here: chambers, alcoves, polarization (see \cite{OS} for an overview). An essential precondition to proving Conjecture \ref{conj:mo} is to properly make these choices such that the map \eqref{eqn:action product} indeed maps $\CA$ into $U_q(\widehat{\mathfrak{g}}^Q)$, although this is straightforward. \\

\subsection{} By comparing the maps \eqref{eqn:stable basis} for $\bm-\boldsymbol{\varepsilon}$ and $\bm+\boldsymbol{\varepsilon}$ (for some $\boldsymbol{\varepsilon} \in \qqp$ very close to $\b0$), the application of the FRT formalism in \cite{OS} yields subalgebras:
\begin{equation}
\label{eqn:subalg mo}
U_q(\mathfrak{g}^Q_{\bm}) \subset U_q(\widehat{\mathfrak{g}}^Q)
\end{equation}
for any $\bm \in \qq$. We expect Conjecture \ref{conj:mo} to match these subalgebras to the slope subalgebras of Subsection \ref{sub:double b}, i.e. that there should be a commutative diagram:
$$
\xymatrix{
\CA  \ar[r]^-\sim & U_q(\widehat{\mathfrak{g}}^Q) \\
\CB_{\bm} \ar@{^{(}->}[u] \ar[r]^-\sim & U_q(\mathfrak{g}^Q_{\bm}) \ar@{^{(}->}[u]}
$$
where the left-most vertical map is prescribed by Proposition \ref{prop:compatible 2}. Moreover, the factorization \eqref{eqn:factorization} should match the analogous factorization of the universal $R$-matrix of $U_q(\widehat{\mathfrak{g}}^Q)$ into the universal $R$-matrices of the subalgebras $U_q(\mathfrak{g}^Q_{\bm+r\bth})$, which is quite tautological in the construction of \cite{OS}. \\

\noindent Of particular interest is the case $\bm = \b0$ of the subalgebra \eqref{eqn:subalg mo}, which is a $q$-deformation of the universal enveloping algebra of the Lie algebra $\mathfrak{g}^Q$ defined by \cite{MO}. Okounkov conjectured that the graded dimension of the latter Lie algebra should be equal to the value of the Kac polynomial of the quiver $Q$. If we knew that $\CB_{\b0} \cong U_q(\mathfrak{g}^Q_{\b0})$, then this conjecture would be equivalent to Conjecture \ref{conj:kac poly}. \\

\end{document}